\documentclass[leqno]{amsart}
\usepackage{amssymb}
\usepackage{amsmath, amsfonts, vmargin, enumerate}
\usepackage{color}
\usepackage{amsthm}
\usepackage[backref]{hyperref}
\usepackage{bbm}
\usepackage{verbatim}
\usepackage{mathrsfs}
\usepackage{txfonts}
\usepackage{cleveref}
\usepackage{upgreek}
\usepackage{manfnt}
\usepackage{bm}

\begin{document}
\newtheorem{theorem}{Theorem}
\newtheorem{proposition}[theorem]{Proposition}
\newtheorem{conjecture}[theorem]{Conjecture}
\newtheorem{corollary}[theorem]{Corollary}
\newtheorem{lemma}[theorem]{Lemma}
\newtheorem{sublemma}[theorem]{Sublemma}
\newtheorem{observation}[theorem]{Observation}
\newtheorem{remark}[theorem]{Remark}
\newtheorem{definition}[theorem]{Definition}
\theoremstyle{definition}
\newtheorem{notation}[theorem]{Notation}
\newtheorem{question}[theorem]{Question}
\newtheorem{example}[theorem]{Example}
\newtheorem{problem}[theorem]{Problem}
\newtheorem{exercise}[theorem]{Exercise}
\numberwithin{theorem}{section} 
\numberwithin{equation}{section}

\def\MM{\mathbb{M}}
\def\NN{\mathbb{N}}
\def\QQ{{\mathbb{Q}}}
\def\RR{{\mathbb R}}
\def\TT{{\mathbb T}}
\def\HH{{\mathbb H}}
\def\ZZ{{\mathbb Z}}
\def\CC{{\mathbb C}}
\def\EE{\mathbb E}
\def\GG{\mathbb G}

\def\HV{\mathbf V}

\def\H{{\mathcal H}}
\def\L{\mathcal{L}}
\def\S{\mathcal{S}}
\def\F{{\mathcal F}}
\def\G{{\mathcal G}}
\def\M{{\mathcal M}}
\def\B{{\mathcal B}}
\def\C{{\mathcal C}}
\def\T{{\mathcal T}}
\def\I{{\mathcal I}}
\def\E{{\mathcal E}}
\def\V{{\mathcal V}}
\def\W{{\mathcal W}}
\def\U{{\mathcal U}}
\def\A{{\mathcal A}}
\def\R{{\mathcal R}}
\def\D{{\mathcal D}}
\def\N{{\mathcal N}}
\def\P{{\mathcal P}}
\def\K{{\mathcal K}}

\def\c{\mathfrak{c}}
\def\g{\mathfrak{g}}
\def\v{\mathfrak{v}}
\def\f{\mathfrak{f}}
\def\q{\mathfrak{p}}

\newcommand{\Norm}[1]{ \left\|  #1 \right\| }
\newcommand{\set}[1]{ \left\{ #1 \right\} }
\newcommand{\lp}[1]{L^{#1}(\GG)}
\newcommand{\dl}[1]{\delta_{#1}}
\newcommand{\wlp}[1]{\L^{#1,\infty}}
\newcommand{\sg}[1]{\sigma_{#1}}
\newcommand{\lgm}[1]{L_{\infty}(\GG,\mathbb{M}_{#1})}
\newcommand{\inner}[1]{ \langle  #1 \rangle}
\newcommand{\bra}[1]{\Bigg(  #1 \Bigg)}
\newcommand{\hw}[1]{\dot{W}^{1,#1}(\GG)}

\def\dist{\operatorname{dist}\,}
\def\p{\partial}
\def\rp{ ^{-1} }
\def\qd{\,{\mathchar'26\mkern-12mu d}}
\def\ri{{\rm i}}
\def\unit{\mathbf 1}
\def\supp{{\rm supp}}
\def\cci{C_{c}^{\infty}(\GG)}
\def\Exp{\operatorname{\rm Exp}}
\def\Log{\operatorname{\rm Log}}
\def\id{{\rm id}}
\def\sgn{{\rm sgn}}
\def\sd{\mathbb{S}^{d-1}}
\def\swq{\dot{W}^{1,\Q}(\GG)}
\def\wh{\tilde}
\def\Tr{{\rm Tr}}
\def\osc{{\rm osc}}
\def\Q{\mathbb Q}

\title{Schatten Properties of Calder\'{o}n--Zygmund Singular Integral Commutator on stratified Lie groups}

\author{Ji Li}

\address{School of Mathematical and Physical Sciences, Macquarie University, Sydney 2109, Australia}
\email{ji.li@mq.edu.au}

\author{Xiao Xiong}
\address{Institute for Advanced Study in Mathematics, Harbin Institute of Technology, 150001 Harbin, China}
\email{xxiong@hit.edu.cn}

\author{Fulin Yang }
\address{School of Mathematics, Harbin Institute of Technology, 150001 Harbin, China}
\email{fulinyoung@hit.edu.cn}




\date{}
\maketitle

\markboth{J. Li, X. Xiong and F.L. Yang}
{Schatten properties}

\begin{abstract}
We provide full characterisation of the Schatten properties of $[M_b,T]$,  the commutator of Calder\'{o}n--Zygmund singular integral $T$ with symbol $b$ $(M_bf(x):=b(x)f(x))$ on stratified Lie groups $\GG$. We show that, when $p$ is larger than the homogeneous dimension $\QQ$ of $\GG$, the Schatten $\L_p$ norm of the commutator is equivalent to the Besov semi-norm $B_{p}^{\QQ/p}$ of the function $b$; but when $p\leq \QQ$, the commutator belongs to $\L_p$ if and only if $b$ is a constant. For the endpoint case at the critical index $p=\QQ$, we further show that the Schatten $\L_{\QQ,\infty}$ norm of the commutator is equivalent to the Sobolev norm $W^{1,\QQ}$ of $b$. Our method at the endpoint case differs from existing methods of Fourier transforms or trace formula for Euclidean spaces or Heisenberg groups, respectively, and hence can be applied to various settings beyond.
\end{abstract}

\section{Introduction}


Schatten class estimates of the Riesz transform commutators link to the quantised derivative of A. Connes \cite{Connes1988, F2022, FSZ2022,LMSZ2017}. A general setting for quantised calculus is a spectral triple $(\A,\H,D)$, which consists of a Hilbert space $\H$, a pre-$C^*$-algebra $\A $, represented faithfully on $\H$ and a self-adjoint operator $D$ acting on $\H$ such that every $a\in A$ maps the domain of $D$ into itself and the commutator $[D,a] = Da-aD$ extends from the domain of $D$ to a bounded linear endomorphism of $\H$. Here, the quantised differential $\qd a$ of $a \in \A$ is defined to be the bounded operator ${\rm i} [{\rm sgn}(D),a]$, ${\rm i}^2=-1$. This is related to the construction of a Fredholm module from a spectral triple, as in e.g. \cite{Connes1995}. For recent progress on quantum differentiability and quantum integral, we refer to \cite{Connes1994,FLMSZ,F2022,FSZ2022,LMSZ2017,LSZ2012,MSX2019,MSX2020,SXZ2023}.

In the above setting of Connes, Schatten properties of compact operators correspond in some way to the ``size'' of the operators, due to the fact that Schatten properties is associated to the rate of decay of the singular values (see \cite{LSZ2012}).
For a compact operator $A$ on some separable Hilbert space $\H$, denote by $\{s(k,A)\}_{k\in\NN}$ its singular values arranged in non-increasing order with multiplicities. 
Of particular interest are those compact operators which satisfy:
    \begin{align*}
	\sum_{n=0}^\infty s(n,A)^p < \infty,\quad\text{ or,} &\quad s(n,A) = O((n+1)^{-1/p}),\quad  n\to \infty,\text{ or,}
	\\
	&\hskip-1.5cm\sup_{n \geq 1} \frac{1}{\log(n+2)} \sum_{k=0}^n s(k,A)^p < \infty\,,
\end{align*}
for some $p\in(0,\infty)$.
The first condition stated above is for $A$ to be in the Schatten ideal $\L_{p}$, the second is for $A$ to be in weak-Schatten ideal $\L_{p,\infty}$, and the final condition is for $|A|^{p}$ to be in the Macaev-Dixmier ideal $\mathcal{M}_{1,\infty}$, (see \cite[Chapter 4, Section 2.$\beta$]{Connes1994} or \cite[Example 2.6.10]{LSZ2012}).

 A model example for quantised calculus is to take a compact Riemannian spin manifold $M$ with Dirac
operator $D$. The algebra $C(M)$ of continuous functions on $M$ acts by pointwise multiplication on $\H$, the Hilbert space of square-integrable sections of a Hermitian vector
bundle. In quantised calculus the immediate question is to determine the relationship between the degree of differentiability of $f\in C(M)$ and the rate of decay of the singular values of $\qd f$. In general, we have the following:
$$f\in C^{\infty}(M)\Rightarrow \qd f\in \L_{d,\infty},$$
where $d$ is the dimension of manilfold $M$ \cite[Theorem~3.1]{Connes1988}.
In this setting, the (weak) Schatten properties of the commutators characterise the quantum differentiability of $f\in C^{\infty}(M)$.

On $\RR^d$, the operator $D$ is indeed the Dirac operator, and thus ${\rm sgn}(D)$ is given by the Riesz transforms. In this case, the study of boundedness or compactness of the commutators goes back to the pioneer works \cite{CRW1976,U1978}, where the authors show that, the commutator is bounded on $L^{2}(\RR^{d})$ if and only if the symbol function $b\in {\rm BMO}(\RR^{d})$, and compact if and only if $b\in {\rm VMO}(\RR^{d})$. Here, ${\rm BMO}(\RR^{d})$ represents the well-known space of bounded mean oscillation, and $ {\rm VMO}(\RR^{d})$ denotes the closure of $C_{c}^{\infty}(\RR^{d})$ in ${\rm BMO}(\RR^{d})$.

Therefore, it is natural to consider the Schatten properties of $[M_{b},T_{0}]$ for general Calder\'{o}n--Zygmund singular integral $T_0$, under the assumption $b\in {\rm VMO}(\RR^{d})$.
When $T_{0}$ is the Hilbert transform on the real line, a combination of the results in Chapter 6 of \cite{Peller2003} asserts that for all $0 < p < \infty$ we have
\[
[M_{b},T_{0}] \in \mathcal{L}_p\Longleftrightarrow b\in B^{\frac1p}_{p }(\mathbb{R}),
\]
where $B_{p}^{\frac{1}{p}}$ is the Besov space on $\RR$.
The original proof of this equivalence goes back to Peller's famous work \cite{Peller1980-R} for $p\geq 1$, and to Peller \cite{Peller1983} and Semmes \cite{Semmes1984} independently for $0< p <1$. 
Higher dimensional cases have been studied in the harmonic analysis literature. Janson and Wolff \cite{JW1982} proved  that 
\begin{enumerate}[(a)]
\item if $b\in {\rm VMO}(\RR^{d})$ and $d<p<\infty$, then $[M_{b},T_{0}]\in \L_{p}$ if and only if $b\in B_{p}^{\frac{d}{p}}$;
\item if $b\in {\rm VMO}(\RR^{d})$ and $0<p\leq d$, then $[M_{b},T_{0}]\in \L_{p}$ if and only if $b$ is a constant,
\end{enumerate}
where $B_{p}^{\frac{d}{p}}$ is now the Besov space on $\RR^d$.
Thereafter, Rochberg and Semmes \cite{RS1989} gave a delicate discussion on $[M_{b},T_{0}]$ and investigated its higher order commutators on $\RR^{d}$.

\medskip

Recently, Schatten class estimates of commutators of Riesz transforms $\R_{k}, k=1,\ldots,2d$, on Heisenberg group $\mathbb{H}^{d}$ have been investigated by Fan, Lacey and Li \cite{FLL2023}:
\begin{enumerate}[(a)]
	\item if $b\in {\rm VMO}(\mathbb{H}^{d})$ and $2d+2<p<\infty$, then $[M_{b},\R_{k}]\in \L_{p}$ if and only if $b\in B_{p}^{\frac{2d+2}{p}}(\mathbb{H}^{d})$,
	\item if $b\in {\rm VMO}(\mathbb{H}^{d})$ and $0<p\leq 2d+2$, then $[M_{b},\R_{k}]\in \L_{p}$ if and only if $b$ is a constant.
\end{enumerate}
The above assertion illustrates that the dimensional index $\frac{2d+1}{p}$ in the sense of  Riemannian geometry is replaced by homogeneous dimension index $\frac{2d+2}{p}$ in the sense of sub-Riemannian geometry.
Later on, via establishing a trace formula on $\HH^d$, Fan, Li, Mcdonald, Sukochev and Zanin \cite{FLMSZ}  obtained the endpoint weak Schatten characterisation:
\begin{equation*}
\mbox{suppose}\quad b\in L^{\infty}(\HH^{d}),\quad \mbox{then}\quad [M_{b},R_{k}]\in\L_{2d+2,\infty}\quad\mbox{ if and only if}\quad b\in\dot{W}^{1,2d+2}.
\end{equation*}

Motivated by the results on $\RR^{d}$ and $\HH^{d}$, we aim to establish Schatten estimates of the Calder\'{o}n--Zygmund singular integral commutator on general stratified Lie groups $\GG$ in this note. We will show that:
\begin{enumerate}[(a)]
	\item if $b\in {\rm VMO}(\GG)$ and $\Q<p<\infty$, then $[M_{b},T]\in \L_{p}$ if and only if $b\in B_{p}^{\frac{\Q}{p}}(\GG)$;
	\item if $b\in {\rm VMO}(\GG)$ and $0<p\leq \Q$, then $[M_{b},T]\in \L_{p}$ if and only if $b$ is a constant;

\item if $b\in {\rm VMO}(\GG)$, then $[M_{b},T]\in\L_{\Q,\infty}$ if and only if $b\in\swq$. 
\end{enumerate}
Here $\Q$ denotes the homogeneous dimension of $\GG$, and $T$ denotes the Calder\'on--Zygmund operators on $\GG$.

Comparing to the existing methods for dealing with Schatten class of commutators in $\RR^d$ and $\HH^d$, the main difficulty here is that the Fourier transform and trace formula on $\GG$ are not available. 
This requires us to introduce new techniques different from those in \cite{JW1982,RS1989,FLL2023,FLMSZ}, especially for the  
%
endpoint case, i.e., the $\L_{\Q,\infty}$ property of the commutators (which is much more subtle). 

For Riesz transform commutators on Euclidean space, noncommutative torus and noncommutative Euclidean space, \cite{LMSZ2017,MSX2019,MSX2020} provided trace formulae by using pseudo-differential operator theory and double operator integral. On Heisenberg groups, \cite{FLMSZ} provided a trace formula for Riesz transform commutators  by using the irreducible representations (Schr\"{o}dinger representations) and double operator integral. From such trace formulae, they deduced the weak Schatten class estimates of the Riesz transform commutators. 
But in the current setting, the pseudo-differential operators and the irreducible representations are obscure on Carnot groups, so we have no idea on how to establish similar trace formulae for commutators of $T$ on $\GG$. 

To overcome these difficulties, we apply conditional expectation and Alpert bases with higher vanishing  moments to locally expand the kernel of $T$, so that the commutator becomes some linear combinations of nearly weakly orthonormal (NWO) sequences \cite{RS1989}. 
Then, combining properties of NWO sequences and the oscillatory characterisation of Sobolev space, we can finally obtain the  weak Schatten class estimates. This is a new idea to treat the weak Schatten class estimate beyond the Fourier transform and trace formula. On Heisenberg groups, one advantage of our method is that the assumption $b\in L^{\infty}(\HH^d)$ in \cite{FLMSZ} can be relaxed to $b\in {\rm VMO}(\HH^d)$. Moreover, our endpoint case at the critical index also provides the missing theory in \cite{CDLWW,CDLWW2} on the quaternionic setting for the commutator of Cauchy--Szeg\"o projections and Riesz transforms.



\medskip

To state our main results, let us give a brief introduction to basic structures and analysis on stratified Lie groups.

\subsection{stratified Lie groups}
\subsubsection{Basic structures}

 The following basic facts are easily found in the literature, see e.g. \cite{BLU2007,FS1982}. A stratified Lie group $\GG$ is a connected and simply connected nilpotent Lie group whose real left-invariant Lie algebra $\g$ admits a direct sum decomposition
\begin{equation*}
\g=\bigoplus_{i=1}^{\tau}\g_{i},\quad {\rm where}\quad [\g_{1},\g_{i}]=\g_{i+1}\quad {\rm for}\quad i\leq \tau-1,\quad [\g_{1},\g_{\tau}]=0.
\end{equation*}
Here $\tau$ is called the step of group $\GG$. One can identify $\g$ with $\GG$ via the exponential map
\begin{align*}
\Exp:\g\rightarrow\GG,
\end{align*}
which is a diffeomorphism and denote by $\Log$ its inverse which is also a diffeomorphism with polynomial component functions \cite[Theorem 1.3.28]{BLU2007}.
Let 
\begin{align*}
d=\sum_{i=1}^{\tau}n_{i}\quad{\rm with}\quad n_{i}={\rm dim}(\g_{i}).
\end{align*}
By \cite[Theorem~2.2.18]{BLU2007}, one identifies $\GG$ with a nilpotent stratified Lie group $(\RR^{d},\circ)$ via some Lie group isomorphism. So in this note, we will focus on $\GG=(\RR^{d},\circ)$. But since no ambiguity is arisen, we will often omit the multiplication operation $\circ$ in the sequel.

Writing $\RR^{d}=\RR^{n_{1}}\times\RR^{n_{2}}\times\cdots\times\RR^{n_{\tau}}$, then $x\in\GG$ admits a decomposition $(x^{(1)},x^{(2)},\ldots,x^{(\tau)})$ for $x^{(k)}\in\RR^{n_{k}}$. There is a nice polynomial formula of the multiplication mapping: For $x,y\in\GG$, denoting $xy$ by $((xy)^{(1)},(xy)^{(2)},\ldots,(xy)^{(\tau)})$, we have
\begin{align*}
&(xy)^{(1)}=x^{(1)}+y^{(1)},\\
&(xy)^{(j)}=x^{(j)}+y^{(j)}+P^{(j)}(x,y),\quad 2\leq j\leq \tau,
\end{align*}
where each $P^{(j)}(x,y)$ is a homogeneous polynomial with degree $j$ depending only on the previous variables $x^{(1)},\ldots,x^{(j-1)}$ and $y^{(1)},\ldots,y^{(j-1)}$.

The dilations on $\GG$, naturally arising from the direct sum decomposition of $\g$, are defined as
$$\delta_{r}(x)=(rx^{(1)},r^{2}x^{(2)},\ldots,r^{\tau}x^{(\tau)}),\quad r>0$$
for $x=(x^{(1)},x^{(2)},\ldots,x^{(\tau)})\in\GG$. For later convenience, we also denote $\delta_{r}(x)$ by $rx$. 
There are many equivalent symmetric homogeneous norms on $\GG$.
For examples, for $x=(x^{(1)},x^{(2)},\ldots,x^{(\tau)})$, we set 
\begin{align*}
\rho(x)=\bigg(\sum_{k=1}^{\tau}|x^{(k)}|^{\frac{2}{k}}\bigg)^{\frac{1}{2}}\quad{\rm and}\quad
\rho_{\infty}(x)=\max_{1\leq k\leq \tau,1\leq j\leq n_{k}}|x^{(k)}_{j}|^{\frac{1}{k}},
\end{align*}
where $|x^{(k)}|$ denotes the Euclidean norm on $\RR^{n_{k}}$.
One easily finds a constant $C_{\rho}>1$ such that
\begin{align}\label{norm-inequality}
\rho_{\infty}(x)\leq\rho(x)\leq  C_{\rho} \rho_{\infty}(x).
\end{align}
Since there is some constant $A_{0}\geq1$ such that
 \begin{align}\label{triangle-inequality}
 	\rho(xy)\leq A_{0}(\rho(x)+\rho(y)),\quad x,y\in\GG,
 \end{align}
we see that $\rho(\cdot)$ is a quasi-distance on $\GG$. For $r>0$, denote $B(x,r)$ the open ball in $\GG$ associated to the quasi-distance:
\begin{equation}\label{ball}
	 B(x,r) =  \{ y \in \GG:  \rho(y\rp x)< r  \}.
\end{equation}

Since the Lebesgue measure on $\RR^d$ is invariant with respect to the left and right translations on $\GG$, the Lebesgue measure on $\RR^d$ is automatically the Haar measure $\mu$ on $\GG$. For a measurable subset $E\subset \GG$, we have
$$\mu(r E)  = r^\Q \mu(E)$$
with the homogeneous dimension $\Q$ given by
$$\Q =\sum_{i=1}^{\tau}i\cdot n_{i}.$$
In this sense, we get a space of homogeneous type $(\GG, \rho(\cdot),  \mu )$ under the notation of Coifman--Weiss \cite{CW1971,CW1977}. The above $\Q$ is called the homogeneous dimension of $(\GG, \rho(\cdot), \mu )$.

If $\Q\leq 3$, then $\GG$ must equal $\RR^\QQ$ with $\circ = +$. Therefore, in this paper, we always assume that $\Q\geq4$. Note that the simplest  non-Euclidean stratified Lie group is the Heisenberg group $\HH^{1}=(\RR^{3},\circ)$ with homogeneous dimension 4.

\subsubsection{Besov and Sobolev spaces}
Let ${{\rm poly}(\GG)}$ be the space of all polynomials and $\S(\GG)$ be the Schwartz function space on $\GG$, see e.g. \cite[Chapter 1]{FS1982}. There is a canonical Fr\'{e}chet topology on $\S(\GG)$, and $\S'(\GG)$, the space of tempered distribution, is defined to be the topological dual of $\S(\GG)$.

For $f\in\S(\GG)$ and $w,y\in\GG$, define $\lambda_{w}f(y)=f(w\rp y)$, and then extend it to $\S'(\GG)$ by duality.
For $0<\alpha<1$, $1\leq p<\infty$, the (homogeneous) Besov space on $\GG$ is defined as
\begin{align*}
B_{p}^{\alpha}=\set{f\in \S'(\GG)/{{\rm poly}(\GG)}: \Norm{f}_{B_{p}^{\alpha}}<\infty},
\end{align*}
where
\begin{align*}
	\Norm{f}_{B_{p}^{\alpha}}:=
\Bigg(\int_{\GG}\frac{\Norm{\lambda_{w}f-f}_{p}^{p}}{\rho(w)^{p\alpha}}\frac{dw}{\rho(w)^{\Q}}\Bigg)^{\frac{1}{p}}.
\end{align*}

The above Besov semi-norm is given in \cite{CRT2001,IY2012} in terms of differences. There are many other equivalent characterisations, see for example: heat semigroup characterisation in \cite{FMV2006} and Littlewood--Paley characterisation in \cite{FMV2006,IY2012}.

Let ${\{X_{j}\}}_{j=1}^{n_{1}}$ be a basis of the smooth vector fields in $\g_{1}$. The horizontal gradient and sub-Laplacian operator on $\GG$ are separately defined as follows
\begin{align*}
\nabla=(X_{1},X_{2},\ldots,X_{n_{1}})\quad{\rm and}\quad
\Delta=-\sum_{j=1}^{n_{1}}X_{j}^{2}.
\end{align*}
The notation $\dot{W}^{1,p}(\GG)$ ($1<p<\infty$) stands for the stratified Lie group $\GG$ of the homogeneous Sobolev spaces. More precisely, the set $\dot{W}^{1,p}(\GG)$ consists of all $f\in\S'(\GG)$ such that $X_{j}f\in\lp{p}$ for all $j=1,2,\ldots,n_{1}$. 
The related homogeneous Sobolev semi-norm is defined as
\begin{align*}
\Norm{f}_{\dot{W}^{1,p}(\GG)}
=\Big(\sum_{j=1}^{n_{1}}\Norm{X_{j}f}_{\lp{p}}^{p}\Big)^{\frac{1}{p}}.
\end{align*}
Then $\dot{W}^{1,p}(\GG)$ ($1<p<\infty$) becomes a quasi-Banach space when equipped with this semi-norm.

\subsubsection{Singular integrals}
Let
\begin{align*}
\R_{j}=X_{j}\Delta^{-\frac{1}{2}},\quad j=1,2,\ldots,n_{1}.
\end{align*}
These operators are called Riesz transforms of $\GG$. 
For $x\in\GG$, denote
\begin{align*}
K_{j}(x)=\frac{1}{\sqrt{\pi}}\int_{0}^{\infty}t^{-\frac{\Q}{2}-1}(X_{j}h)(t^{-\frac{1}{2}} x)dt,
\end{align*} 
where $h$ is the $C^{\infty}$ solution of the heat equation $(\frac{\p}{\p t}+\Delta)u=0$ ( see  \cite[chapter 1]{FS1982}, \cite[IV.4]{VSC1992} ). 
By \cite{FS1982}, $K_{j}$ is smooth out of the origin $o$ and is homogeneous of degree $-\Q$, i.e.
$$K_{j}(rx)=r^{-\Q}K_{j}(x) \quad \mbox{for} \quad x\neq o.$$
For any $j\in\{1,2,\ldots,n_{1}\}$, the $j$-th Riesz transform $\R_{j}$ is a convolutional singular integral operator with the kernel $K_{j}$, i.e.
\begin{align*}
	\R_{j}f(x)=\int_{\GG}K_{j}(y\rp x)f(y)dy\quad{\rm for}\quad f\in \S(\GG).
	\end{align*}
Here and in the sequel, we abbreviate the Haar measure $d\mu(y)$ on $\GG$ as $dy$.

Moreover, we will consider general convolutional singular integral operator on $\GG$.
Let $K:\GG\setminus\set{o}\rightarrow\RR$ and denote $K(x,y)=K(y\rp x)$. Suppose that
\begin{enumerate}[(i)]
	\item $|K(x,y)|\leq\frac{C}{\rho(y^{-1}x)^{\Q}}$, for $x\neq y$ and some constant $C>0$;
	\item if $\frac{\rho(x_{1}^{-1}x)}{\rho(y^{-1}x)}\leq\frac{1}{2A_{0}}$, then there is some $\sigma>0$ and constant $C>0$ such that
	\begin{align*}
		|K(x,y)-K(x_{1},y)|+|K(y,x)-K(y,x_{1})|
\leq C\frac{\rho(x_{1}^{-1}x)^{\sigma}}{\rho(y^{-1}x)^{\Q+\sigma}}.
	\end{align*}
\end{enumerate}
If $K$ satisfies these two conditions, then
\begin{align*}
	Tf(x)=\int_{\GG}K(y\rp x)f(y)dy,\quad \mbox{for}\quad f\in\S(\GG), 
\end{align*}
gives a  Calder\'{o}n--Zygmund singular integral operator on $\GG$.

In this note, in order to obtain the Schatten norm characterisation of singular integral commutators on $\GG$, we will additionally make the following ``{\it non-degenerate condition }'' assumption on the kernel $K$:\\
There are positive constants $A _{4}\geq A _{3}\geq2A_{0}$ such that, for each ball $B(x_{0},r)$, one can find another ball $B(y_{0},r)$ with  $A _{3}r\leq\rho(x_{0}\rp y_{0})\leq A _{4}r$ satisfying that, for all $(x,y)\in B(x_{0},r)\times B(y_{0},r)$,
\begin{equation}\label{not-change-sign}
	K(x,y) \mbox{ does not change sign }
\end{equation}
and
\begin{equation}\label{nondegenerate}
	 |K(x,y)|\gtrsim\frac{1}{r^{\Q}} .
\end{equation}

Such non-degenerate condition was first proposed in \cite{H2012} and then studied in \cite{LOR2017} in $\RR^{n}$ and \cite{DGKL2020} in the space of homogeneous type.
In classical case, the assumptions \eqref{not-change-sign} and \eqref{nondegenerate} are automatic for Riesz transforms, and \eqref{nondegenerate} is automatic for homogeneous singular integral with smooth kernel $\frac{\Omega(x)}{|x|^{n}}$ for $\Omega$ with mean value zero on the unit sphere.

For stratified Lie groups, Riesz transforms $\R_{j}$ ($j=1,2,\ldots,n_{1}$) are Calder\'{o}n--Zygmund singular integral operator with standard convolutional kernel for $\delta=1$. In this setting, \eqref{not-change-sign} and \eqref{nondegenerate} were verified in \cite{DLLW2019} for the kernel $K_{j}$.

\subsection{Schatten class}
The following material is standard; for more details we refer to \cite{Peller2003,Simon1979}.
Let $\H$ be a Hilbert space. Denote $\B(\H)$ the set of all bounded linear operators on $\H$ and $\K(\H)$ the ideal of compact operators on $\H$. Define the absolute value of $A$ by 
$$|A|=\sqrt{A^{*}A}.$$
Given $A\in\K(\H)$, the sequence of singular values $s(A)=\{s(k,A)\}_{k\in\NN}$ is defined as follows,
\begin{align*}
s(k,A)=\inf\{\|A-F\|:\quad {\rm Rank}(F)\leq k,\quad F\in\B(\H) \}.
\end{align*}
Equivalently, $s(A)$ is the sequence of eigenvalues of $|A|$ arranged in non-increasing order with multiplies. 


Let $0<p,q<\infty$. The Schatten class $\L_{p,q}(\H)$ is the set of operators $A\in\K(\H)$ such that $\{s(k,A)\}_{k\in\NN}$ is $\ell^{p,q}$-summable, i.e. in the Lorentz-Lebesgue sequence space $\ell^{p,q}$. The $\L_{p,q}(\H)$-norm is defined by
\begin{align*}
\|A\|_{\L_{p,q}(\H)}=\Norm{s(A)}_{\ell^{p,q}}
=\Big(\sum_{k=0}^{\infty}s(k,A)^{q}(1+k)^{\frac{q}{p}-1}\Big)^{\frac{1}{q}}.
\end{align*}
We will simply write $\L_{p}(\H)=\L_{p,p}(\H)$.
With this norm $\L_{p}(\H)$ is a Banach space and even more an ideal of $\B(\H)$ when $p\geq1$.
On $\lp{2}$, the Schatten $p$-norm ($p\geq2$) can also be characterized as follows,
\begin{align*}
\|A\|_{\L_{p}(\lp{2})}=\sup\bigg\{\Big(\sum_{k\in\NN}\|Ae_{k}\|^{p}_{\lp{2}}\Big)^{\frac{1}{p}}:\quad \{e_{k}\}~ {{\rm is~ an~ orthonormal~ base~ in}~\lp{2}}\bigg\}.
\end{align*}

The weak Schatten class $\L_{p,\infty}(\H)$ is the set of operators $A$ such that $s(A)$ is in the weak sequence $L_{p}$-space $\ell^{p,\infty}$, with quasi-norm:
\begin{align*}
\|A\|_{\L_{p,\infty}(\H)}=\Norm{s(A)}_{\ell^{p,\infty}}=\sup_{k\geq0}(k+1)^{\frac{1}{p}}s(k,A)<\infty.
\end{align*}
As with the $\L_{p}(\H)$ spaces, $\L_{p,\infty}(\H)$ is an ideal of $\B(\H)$.

Later on, $\L_{p}$ and $\L_{p,\infty}$ always denote the Schatten class and the weak Schatten class on $\lp{2}$ except for special explanations.

\subsection{Main results}

Let $M_{b}$, for $b\in L_{loc}^{1}(\GG)$, be the multiplication operator defined as
\begin{align*}
M_{b}f(x)=b(x)f(x)\quad\mbox{for}\quad f\in L_{loc}^{1}(\GG),
\end{align*}
and let $T$ be a Calder\'{o}n--Zygmund singular integral operator. The commutator of $M_{b}$ with $T$ is defined as follows
\begin{align*}
[M_{b},T]=M_{b}T-TM_{b}.
\end{align*}

Recall the BMO space on $\GG$, following \cite[chapter 5.B]{FS1982},
\begin{align*}
{\rm BMO}(\GG)=\set{b\in L_{loc}^{1}(\GG):\Norm{b}_{BMO}<\infty},
\end{align*}
where
\begin{align*}
\Norm{b}_{{\rm BMO}(\GG)}=\sup_{B}\frac{1}{|B|}\int_{B}|b(x)-b_{B}|dx.
\end{align*}
Here  $b_{B}=\frac{1}{|B|}\int_{B}b(y)dy$, and the supremum runs over all balls $B\subset\GG$ of the form \eqref{ball}.
The VMO space ${\rm VMO}(\GG)$ is defined to be the closure of $\cci$ functions on $\GG$ with respect to the ${\rm BMO}(\GG)$ norm.

In \cite{DLLW2019}, the authors showed that $[M_{b},T]\in\B(L^{2}(\GG))$ if and only if $b\in {\rm BMO}(\GG)$. 
In \cite{CDLW2019}, the authors proved that $[M_{b},T]\in\K(L^{2}(\GG))$ if and only if $b\in {\rm VMO}(\GG)$.

The following theorem is the Schatten $\L_{p}$ property of $[M_{b},T]$ on $\GG$.
\begin{theorem}\label{mt}
Let $T$ be a Calder\'{o}n--Zygmund singular integral operator satisfying \eqref{not-change-sign} and \eqref{nondegenerate}.
Assume that $b\in {\rm VMO}(\GG)$ and $0<p<\infty$.
Then $[M_{b},T]\in\L_{p}$ if and only if 
\begin{enumerate}[{\rm (a)}]
\item \label{mta} $b\in B_{p}^{\frac{\Q}{p}}$, if $\frac{\Q}{p}<1$; in this case we have $\Norm{[M_{b},T]}_{\L_{p}}\simeq\Norm{b}_{B_{p}^{\frac{\Q}{p}}}$.
    \item\label{mtb} $b$ is a constant, if $\frac{\Q}{p}\geq1$.
\end{enumerate}
\end{theorem}

After obtaining the cut-off, it is natural to consider the endpoint case. For our purpose, we require some differentiable conditions on $K$ besides  \eqref{not-change-sign} and \eqref{nondegenerate} in ``non-degenerate condition''. For $\gamma_{T}\in\NN$, we say that kernel $K$ is differentiable up to $\gamma_{T}$-th order if
\begin{align*}
|\nabla_{y}^{\beta}\nabla_{x}^{\alpha}K(x,y)|\simeq\frac{1}{\rho(x,y)^{\Q+|\alpha|+|\beta|}},
\quad|\alpha|,|\beta|\leq \gamma_{T},
\end{align*}
for $\alpha,\beta\in\NN^{n_{1}}$, $x\neq y\in\GG$ and some constant depending on $K$ and $\GG$.  
On $\GG$, we have the following characterisation of weak Schatten class estimate for $[M_{b},T]$.
\begin{theorem}\label{endponit-estimates}
Let $T$ be a Calder\'{o}n--Zygmund singular integral operator satisfying \eqref{not-change-sign} and \eqref{nondegenerate} with kernel $K$ differentiable up to $\gamma_{T}$-th order for some $\gamma_{T}>\Q$.
Assume that $b\in {\rm VMO}(\GG)$. 
Then the commutator $[M_{b},T]\in\L_{\Q,\infty}$ if and only if $b\in\swq$. 
More precisely, there are positive constants $C$ and $c$ such that
\begin{align*}
c\Norm{b}_{\swq}\leq \Norm{[M_{b},T]}_{\L_{\Q,\infty}}\leq C\Norm{b}_{\swq}.
\end{align*}
\end{theorem}

As a corollary of \Cref{mt} and \Cref{endponit-estimates}, it gives the Schatten estimates of $[M_{b},\R_{k}]$ on $\GG$.
\begin{corollary}\label{rieszcommutator}
Let $j\in\{1,2,\ldots,n_{1}\}$. We have
\begin{enumerate}[(a)]
	\item if $b\in {\rm VMO}(\GG)$ and $\Q<p<\infty$, then $[M_{b},\R_{j}]\in \L_{p}$ if and only if $b\in B_{p}^{\frac{\Q}{p}}(\GG)$;
	\item if $b\in {\rm VMO}(\GG)$ and $0<p\leq \Q$, then $[M_{b},\R_{j}]\in \L_{p}$ if and only if $b$ is a constant;

\item if $b\in {\rm VMO}(\GG)$, then $[M_{b},\R_{j}]\in\L_{\Q,\infty}$ if and only if $b\in\swq$. 
\end{enumerate}

\end{corollary}

This paper is organized as follows. 
In \Cref{mt-mta}, we provide the proof of \Cref{mt} \eqref{mta} in which the upper bound is established in \Cref{mt-mta-upper} and lower bound is established in \Cref{mt-mta-lower}. 
In \Cref{mt-mtb}, we prove  \Cref{mt} \eqref{mtb}. 
In \Cref{endponit-estimates-sect}, we show \Cref{endponit-estimates} in which the sufficiency of \Cref{endponit-estimates} is arranged in \Cref{endponit-estimates-sect-su} and the necessity of \Cref{endponit-estimates} is arranged in \Cref{endponit-estimates-sect-ne}.
As an application, we obtain endpoint weak Schatten estimate for the commutators of the Cauchy--Szeg\"{o} projection on quaternionic Siegel upper half space and the Riesz transforms on quaternionic  Heisenberg groups in \Cref{applications}, which provides the missing theory in \cite{CDLWW,CDLWW2}.
Lastly, the equivalent characterisation of Sobolev space with respect to mean oscillation is arranged in \Cref{Appendix}.

Throughout the paper, the indicator function of a subset $E\subset \GG$ is denoted by $1_{E}$. We write $A\lesssim B$ if $A\leq CB$ for some constant $C>0$ which does not depend on $A$ and $B$, and $A\simeq B$ denote
the statement that $A\lesssim B$ and $B\lesssim A$.

\section{Proof of \Cref{mt} \eqref{mta}}\label{mt-mta}

In this section, we are going to prove Theorem \ref{mt} \eqref{mta}. The main tools are the cube system and Haar basis introduced in \cite{HK2012,KLPW2016}, and the notion of nearly weakly orthonormal (NWO) sequences of functions proposed in \cite{RS1989}.

\medskip

{\it Dyadic cube system.}
Following \cite{HK2012,KLPW2016}, a countable family $\D:=\bigcup_{k\in\ZZ}\D_{k}$ of Borel sets on $\GG$ is called a system of dyadic cubes with parameters $0<\gamma_{1}\leq\gamma_{2}<\infty$ if it has the following properties:
\begin{enumerate}[{\rm (a)}]
\item $\GG=\bigcup_{R\in\D_{k}}R$ (disjoint union) for $k\in\ZZ$;

\item when $k\geq l$ and $R\in\D_{k}, \tilde{R}\in\D_{l}$, one has either $R\subset\tilde{R}$ or $R\cap\tilde{R}=\varnothing$;

\item for $k\in\ZZ$, each cube $R\in\D_{k}$ is a disjoint union of almost $2^{\Q}$ children in $\D_{k+1}$;

\item for $k\in\ZZ$ and $R\in\D_{k}$, there are positive constants $\gamma_{1}, \gamma_{2}$ such that
$$B(c_{R}, \gamma_{1}2^{-k})\subset R\subset B(c_{R}, \gamma_{2}2^{-k}),$$ where $c_{R}$ denotes the center of $R$ and we will use this notion throughout this paper;

\item if $k\leq l$ and $R\in\D_{l},\tilde{R}\in\D_{k}$ with $R\subset\tilde{R}$, then $B(c_{R}, \gamma_{2}2^{-k})\subset B(c_{\tilde{R}}, \gamma_{2}2^{-k})$.

\end{enumerate}
Here $\D_{k}$ is called the k-th dyadic cube system whose cube has side length $2^{-k}$.
More precisely, the k-th dyadic cube system has the following form
\begin{align*}
\D_{k}=\Big\{\delta_{2^{-k}}(g)\cdot\delta_{2^{-k}}(\Omega):g\in\Gamma\Big\},
\end{align*}
where $\Omega$ is a bounded set and $\Gamma$ is a lattice set associated to $\Omega$. 
The existence of such cube system on $\GG$ can be found in \cite{ADM2023,CT2016,HK2012,KLPW2016,RSS1992}.

For $k\in\ZZ$ and $b\in L_{loc}^{1}(\GG)$, the conditional expectation on the cube system $\D$ is defined by
\begin{align*}
\EE_{k}(b)(x)=\sum_{R\in\D_{k}}\EE_{R}(b)(x)
\quad\mbox{where}\quad
\EE_{R}(b)(x)=b_{R}1_{R}(x).
\end{align*}
Here $b_{R}=\fint_{R}b(y)dy$ denote the integral mean of function $b$ over the cube $R\in\D$ with respect to the Lebesgue measure.
Then for every $b\in\lp{p}$, $1<p<\infty$, there holds
\begin{align*}
\EE_{k}(b)\rightarrow b~{\rm as}~ k\rightarrow\infty,\quad
\EE_{k}(b)\rightarrow 0~{\rm as}~ k\rightarrow-\infty.
\end{align*}
The convergence takes place both in the $L^{p}(\GG)$-norm and pointwise almost everywhere.

Given such a dyadic cube system, we are able to restate the non-degenerate conditions \eqref{not-change-sign} and \eqref{nondegenerate} with respect to this dyadic cube system, that will be more convenient for later use. 
	\begin{lemma}\label{cube-nondegenerate}
		Let $\D$ be a dyadic cube system on $\GG$ and $\D_{k}$ be the corresponding k-th dyadic cube system.
		There are positive constants $A'_{4}\geq A'_{3}\geq2A_{0}$ such that, for each cube $R\in\D_{k}$ (with center $c_{R}$), one can find another cube $\wh{R}\in\D_{k}$ (with center $c_{\wh{R}}$) with  $A'_{3}2^{-k}\leq\rho(c_{R}\rp c_{\wh{R}})\leq A'_{4}2^{-k}$ satisfying that, for all $(x,y)\in R\times \wh{R}$,
	\begin{equation}\label{cube-not-change-sign}
		K(x,y) \mbox{ does not change sign }
	\end{equation}
	and
	\begin{equation}\label{cube-lower-bd}
	|K(x,y)|\gtrsim|R|\rp .
	\end{equation}
	
	\end{lemma}

	\begin{proof}
		For $R\in\D_{k}$, one has $R\subset B(c_{R},\gamma_{2}2^{-k})\subset B(c_{R},2A_{0}(1+\gamma_{2})2^{-k})$.
		It follows from \eqref{not-change-sign} and \eqref{nondegenerate} that there is another ball $B(\wh{y},2A_{0}(1+\gamma_{2})2^{-k})$ with $A _{3}2A_{0}(1+\gamma_{2})2^{-k}\leq\rho(c_{R}\rp \wh{y})\leq A _{4}2A_{0}(1+\gamma_{2})2^{-k}$
		satisfying that, for all $(x,y)\in B(c_{R},2A_{0}(1+\gamma_{2})2^{-k})\times B(\wh{y},2A_{0}(1+\gamma_{2}) 2^{-k})$,
		\begin{enumerate}[(a)]
			\item  $K(x,y)$ does not change sign; 
			\item  $|K(x,y)|\gtrsim\frac{1}{(2A_{0}(1+\gamma_{2})2^{-k})^{\Q}}$.
		\end{enumerate}
		One can pick $\wh{R}\in\D_{k}$ such that $\wh{y}$ is contained in the closure of $\wh{R}$, so $\rho(\wh{y}\rp c_{\wh{R}})\leq\gamma_{2}2^{-k}$.
		Thus, for each $w\in\wh{R}$,
		\begin{align*}
			\rho(w\rp \wh{y})\leq A_{0}\Big(\rho(w\rp c_{\wh{R}})+\rho(c_{\wh{R}}\rp\wh{y})\Big)
			\leq 2A_{0}(1+\gamma_{2})2^{-k}.
		\end{align*}
		In other words, $\wh{R}\subset B(\wh{y},2A_{0}(1+\gamma_{2})2^{-k})$.
		Moreover,
		\begin{align*}
			\rho(c_{R}\rp c_{\wh{R}})
			\geq \frac{1}{A_{0}}\rho(c_{R}\rp \wh{y})-\rho(\wh{y}\rp c_{\wh{R}})
			\geq 2A _{3}(1+\gamma_{2})2^{-k}- \gamma_{2}2^{-k}
			\geq 2A_{0}\gamma_{2}2^{-k}, 
		\end{align*}
	and
		\begin{align*}
			\rho(c_{R}\rp c_{\wh{R}})
			\leq A_{0}\Big(\rho(c_{R}\rp \wh{y})+\rho(\wh{y}\rp c_{\wh{R}})\Big)
			\leq (A _{4}+1)2A_{0}^{2}(1+\gamma_{2})2^{-k}.
		\end{align*}
		Letting $A'_{3}=2A_{0}\gamma_{2}$ and $A'_{4}=(A _{4}+1)2A_{0}^{2}(1+\gamma_{2})$, we get the desired result.
	\end{proof}

\medskip

{\it Adjacent dyadic cube system.}
On $\GG$, a finite collection $\{\D^{t}:t=1,2,\ldots,\T\}$ of the dyadic families is called a collection of adjacent system of dyadic cubes with parameters $0<\gamma_{1}\leq\gamma_{2}<\infty$ and $1\leq C_{adj}<\infty$ if it has the following properties:
\begin{enumerate}[{\rm (a)}]
\item each $\D^{t}$ is a system of dyadic cubes with parameters $0<\gamma_{1}\leq\gamma_{2}<\infty$;
    
\item for each ball $B(x,r)\subset\GG$ with $2^{-k-3}< r\leq 2^{-k-2}$, $k\in\ZZ$, there exist $t\in\{1,2,\ldots,\T\}$ and $R'\in\D_{k}^{t}$ with center point $c_{R'}$ such that $\rho(x\rp c_{R'})\leq 2A_{0}2^{-k}$ and $$B(x,r)\subset R'\subset B(x,C_{adj}r).$$
\end{enumerate}
The existence of the adjacent dyadic cube systems on $\GG$ we refer to \cite{HK2012,KLPW2016}. 
Moreover, let $t\in\{1,2,\ldots,\T\}$ and denote $\EE^{t}_{k}$ the conditional expectation associated to $\D^{t}_{k}$ of the cube system $\D^{t}$.

\medskip

{\it Haar basis.}
Let $t\in\{1,2,\ldots,\T\}$.
Regarding $\GG$ as a space of homogeneous-type, \cite{KLPW2016} gives the explicit construction of a Haar basis. Let $n_{\Q}=2^{\Q}-1$. Denote 
$$\set{h_{R}^{t,j}:R\in\D^{t}, j=1,2,\ldots,n_{\Q}}$$ the Haar basis for $\lp{p}$ ($1<p<\infty$) with respect to the cube system $\D^{t}$ on $\GG$. The Haar functions have the following basic properties:
\begin{enumerate}[{\rm (a)}]
	\item $h_{R}^{t,j}$ is a simple Borel-measurable real function on $\GG$;
	
	\item support of $h_{R}^{t,j}$ is $R$;
	
	\item $h_{R}^{t,j}$ is constant on each children of $R$;
	
	\item integral of $h_{R}^{t,j}$ on $\GG$ is vanishing;
	
	\item if $j',j\in\{1,2,\ldots,n_{\Q}\}$ and $j'\neq j$, then $h_{R}^{t,j'}$ and $h_{R}^{t,j}$ are orthogonal;
	
	\item the collection $$\{|R|^{-\frac{1}{2}}1_{R}\}\cup\set{h_{R}^{t,j}:R\in\D^{t}, j=1,2,\ldots,n_{\Q}}$$ is an orthogonal basis for the vector space $V(R)$ of all functions on $R$ that are constant on each sub-cube of $R$;
	
	\item for $1\leq p\leq\infty$, one finds $\|h_{R}^{t,j}\|_{L^{p}(\GG)}\simeq |R|^{\frac{1}{p}-\frac{1}{2}}$.
\end{enumerate}

If $1<p<\infty$ and $f\in\lp{p}$, one has
\begin{align*}
f(x)=\sum_{R\in\D^{t}}\sum_{j=1}^{n_{\Q}}\inner{f,h_{R}^{t,j}}h_{R}^{t,j}(x),
\end{align*}
where the sum converges (unconditionally) both in $\lp{p}$ and pointwise almost everywhere,
see e.g. \cite{KLPW2016}.

\medskip

{\it NWO sequence.}
In \cite{RS1989}, the authors proposed the terminology of nearly weakly orthonormal sequence and then apply this terminology to estimate the Schatten $p$-norm of singular commutator on $\RR^{n}$. This notation is closely connected to Carleson measures. For our purposes, we do not need to recall the full definition, but just recall the following lemmas that will be useful. The first one is the verification of a NWO sequence.
\begin{lemma}\label{NWO}
If functions $\{e_{R}\}_{R\in\D}$ with $\supp(e_{R})\subset R$ satisfy $\Norm{e_{R}}_{q}\leq|R|^{\frac{1}{q}-\frac{1}{2}}$ for some $q>2$, then $\{e_{R}\}_{R\in\D}$ is a NWO sequence.
\end{lemma}

From the point in \cite{R1993,RS1989}, NWO sequence provides a nice finite dimensional approximation for estimating singular values of a compact operator. 
\begin{lemma}\label{compact-upper-bd}
Let $0<p<\infty$, $0<q\leq\infty$ and $\{\lambda_{R}\}_{R\in\D}\in\ell^{p,q}$. Suppose that $\{e_{R}\}_{R\in\D}$ and $\{f_{R}\}_{R\in\D}$ are NWO sequences, and $\displaystyle A=\sum_{R\in\D}\lambda_{R} \inner{\cdot,e_{R}}f_{R}$ is a compact operator on $\lp{2}$. Then
\begin{align*}
\Norm{A}_{\L_{p,q}}\leq C\Norm{\{\lambda_{R}\}_{R\in\D}}_{\ell^{p,q}}.
\end{align*}
\end{lemma}

By \cite[Note 2 (Theorem 6.4)]{R1993}, it implies below result.
\begin{lemma}\label{compact-upper-bd-function}
Suppose that $\{|e_{R}|\}_{R\in\D}$ and $\{|f_{R}|\}_{R\in\D}$ are NWO sequences, and $\displaystyle A\varphi(x)=\sum_{R\in\D}\lambda_{R} \inner{\varphi,b_{R}(\cdot,x)e_{R}}f_{R}(x)$ is a compact operator on $\lp{2}$ with $|b_{R}(x,y)|\leq1$.
If $2<p<\infty$, $1\leq q\leq\infty$ and $\{\lambda_{R}\}_{R\in\D}\in\ell^{p,q}$,  then
\begin{align*}
\Norm{A}_{\L_{p,q}}\leq C\Norm{\{\lambda_{R}\}_{R\in\D}}_{\ell^{p,q}}.
\end{align*}
\end{lemma}

And, we can extract the following result from \cite{RS1989}.
\begin{lemma}\label{compact-upper-bd-sequence}
Let $0<p<\infty$, $0<q\leq\infty$ and $\{\lambda_{R}\}_{R\in\D}\in\ell^{p,q}$. Suppose that $\{G_{R,l}\}_{R\in\D}$ and $\{F_{R,l}\}_{R\in\D}$ are NWO sequences. If $\displaystyle A=\sum_{l\in\ZZ^{2}}\sum_{R\in\D}\lambda_{R}\gamma_{R,l} \inner{\cdot,G_{R,l}}F_{R,l}$ is a compact operator on $\lp{2}$ with $|\gamma_{R,l}|\lesssim\frac{1}{(1+|l|)^{r}}$ for some $r\geq1$, then
\begin{align*}
\Norm{A}_{\L_{p,q}}\leq C\Norm{\{\lambda_{R}\}_{R\in\D}}_{\ell^{p,q}}.
\end{align*}
\end{lemma}

The last one is an estimate of $\L_{p}$-norm for a compact operator with respect to a NWO sequence (the statement on $\RR^{n}$ in \cite{RS1989}).
\begin{lemma}\label{compact}
Let $1<p<\infty$ and $\{f_{R}\}$, $\{e_{R}\}$ be NWO sequences on $\lp{2}$. If $A$ is a compact operator in $\B(L^{2}(\GG))$, then
\begin{align*}
\Big(\sum_{R\in\D}|\inner{Af_{R},e_{R}}|^{p}\Big)^{\frac{1}{p}}\lesssim\Norm{A}_{\L_{p}}.
\end{align*}
\end{lemma}

\subsection{upper bound}\label{mt-mta-upper}
Let $t\in\{1,2,\ldots,\T\}$.
For $R\in\D^{t}_{k}$, let $h^{t}_{R}$ be the haar function among $\{h_{R}^{t,j}\}_{j=1}^{n_{\Q}}$ satisfying that 
$|\int_{R}b(y)h_{R}^{t,j}(y)dy|$ is maximal for $j\in\{1,2,\ldots,n_{\Q}\}$.
Noting that the function $(\EE^{t}_{k+1}(b)-\EE^{t}_{k}(b))1_{R}$ is a sum of $n_{\Q}$ Haar functions, we are in a finite dimensional setting and all $\lp{p}$-spaces have comparable
norms. Therefore,
\begin{align}\label{equinorm}
\Bigg(\fint_{R}|\EE^{t}_{k+1}(b)(y)-\EE^{t}_{k}(b)(y)|^{p}dy\Bigg)^{\frac{1}{p}}\simeq|R|^{-\frac{1}{2}}|\int_{R}b(y)h^{t}_{R}(y)dy|.
\end{align}

\begin{lemma}\label{twochi}
Assume that $T$ is a Calder\'{o}n--Zygmund singular integral operator satisfying \eqref{not-change-sign} and \eqref{nondegenerate}, $b\in {\rm VMO}(\GG)$ and $R\in\D^{t}_{k}$. Then there are four sets $F_{j}^{R}$ and $E_{j}^{R}$ $(j=1,2)$ such that
$$|R|^{-\frac{1}{2}} \,\Big|\int_{R}b(y)h^{t}_{R}(y)dy\Big|\lesssim\sum_{j=1}^{2}|\inner{[M_{b},T](f_{F_{j}^{R}}), e_{E_{j}^{R}}}|,$$
where
\begin{align*}
f_{F_{j}^{R}}:=|R|^{-\frac{1}{2}}1_{F_{j}^{R}}\quad{\rm and}\quad
e_{F_{j}^{R}}:=|R|^{-\frac{1}{2}}1_{E_{j}^{R}}.
\end{align*}
\end{lemma}
\begin{proof}
By assumption, for $R\in\D^{t}_{k}$, we find a cube $\tilde{R}\in\D^{t}_{k}$ satisfying \eqref{cube-not-change-sign} and \eqref{cube-lower-bd}.
Pick a real number $m(b)$ such that
$$\left|\set{y\in\tilde{R}: b(y)> m(b)}\right|\leq\frac{1}{2}|R|\quad {\rm and}\quad
\left|\set{y\in\tilde{R}: b(y)< m(b)}\right|\leq\frac{1}{2}|R|.$$
This number $m(b)$ is called a median value; it always exists, but may not be unique (see \cite{Journe1983}).
Let $F_{1}^{R}$ and $ F_{2}^{R}$ be two measurable disjoint subsets of $\tilde{R}$ such that $F_{1}^{R}\cup F_{2}^{R}=\tilde{R}$,
\begin{align*}
F_{1}^{R}\subset \set{y\in\tilde{R}: b(y)\leq m(b)}\quad {\rm and}\quad F_{2}^{R}\subset\set{y\in\tilde{R}: b(y)\geq m(b)}
\end{align*}
and that
$|F_{1}^{R}|=\frac{1}{2}|R|~ {\rm and}~ F_{2}^{R}=\frac{1}{2}|R|$.
Moreover, set
\begin{align*}
E_{1}^{R}=\set{x\in R: b(x)\geq m(b)}\quad {\rm and}\quad E_{2}^{R}=\set{x\in R: b(x)\leq m(b)}.
\end{align*}
Note that, for $(x,y)\in E_{j}\times F_{j}$ ($j\in\{1,2\}$), we have
\begin{equation}
	\label{not-change-sign-b} 
	b(x)-b(y)\; \mbox{ does not change sign}
\end{equation}
and
\begin{equation}\label{difference-mean}
	 \quad|b(x)-m(b)|\leq |b(x)-b(y)|.
\end{equation}

Employing the facts $\Norm{h^{t}_{R}}_{\infty}\simeq|R|^{-\frac{1}{2}}$ and $R=E_{1}^{R}\cup E_{2}^{R}$, we obtain
\begin{align*}
\left|\int_{R}b(x)h^{t}_{R}(x)dx\right|&=\left|\int_{R}(b(x)-m(b))h^{t}_{R}(x)dx\right|\\
&\lesssim|R|^{-\frac{1}{2}}\int_{R}|b(x)-m(b)|dx\\
&=|R|^{-\frac{1}{2}}\int_{E_{1}^{R}}|b(x)-m(b)|dx+|R|^{-\frac{1}{2}}\int_{E_{2}^{R}}|b(x)-m(b)|dx.
\end{align*}
It suffices to estimate the last two integrals. For $j\in\{1,2\}$, we deduce that
\begin{align*}
\int_{E_{j}^{R}}|b(x)-m(b)|dx
&=\int_{E_{j}^{R}}\frac{2|F_{j}^{R}|}{|R|}|b(x)-m(b)|dx
=2\int_{E_{j}^{R}}\int_{F_{j}^{R}}\frac{1}{|R|}dy|b(x)-m(b)|dx\\
&\lesssim\int_{E_{j}^{R}}\int_{F_{j}^{R}}|b(x)-m(b)||K(x,y)|dydx,
\end{align*}
where we use \eqref{cube-lower-bd} for the last inequality.
Moreover, by \eqref{not-change-sign-b}, \eqref{difference-mean} and \eqref{cube-not-change-sign}, we have
\begin{align*}
	\int_{E_{j}^{R}}\int_{F_{j}^{R}}|b(x)-m(b)||K(x,y)|dydx
&\leq \int_{E_{j}^{R}}\int_{F_{j}^{R}}|b(x)-b(y)||K(x,y)|dydx\\
&=\left|\int_{E_{j}^{R}}\int_{F_{j}^{R}}(b(x)-b(y))K(x,y)dydx\right|\\
&=|\inner{[M_{b},T](1_{F_{j}^{R}}),1_{E_{j}^{R}}}|\\
&= |R|\;|\inner{[M_{b},T](f_{F_{j}^{R}}), e_{E_{j}^{R}}}| .
\end{align*}
This gives our desired result.
\end{proof}

\begin{lemma}\label{tail}
Assume that $T$ is a Calder\'{o}n--Zygmund singular integral operator satisfying \eqref{not-change-sign} and \eqref{nondegenerate}, and that $b\in {\rm VMO}(\GG)$. If $p\in(1,\infty)$ and $[M_{b},T]\in\L_{p}$, then
\begin{align*}
\Norm{b-\EE^{t}_{N+1}(b)}_{p}\lesssim2^{-\frac{(N+1)\Q}{p}}\Norm{[M_{b},T]}_{\L_{p}},
\end{align*}
where the relevant constant does not depend on $N$.
\end{lemma}
\begin{proof}
Note that  $\EE^{t}_{k}(b)$ tends to $b$ in the sense of $L^{p}$-norm when $k$ tends to $\infty$.
Using dyadic decomposition on $\GG$, we have
\begin{align*}
\Norm{b-\EE^{t}_{N+1}(b)}_{p}
&\leq\sum_{k=N+1}^{\infty}\Norm{\EE^{t}_{k+1}(b)-\EE^{t}_{k}(b)}_{p}
=\sum_{k=N+1}^{\infty}\Bigg(\sum_{R\in\D^{t}_{k}}2^{-k\Q}\fint_{R}|\EE^{t}_{k+1}(b)(y)-\EE^{t}_{k}(b)(y)|^{p}dy\Bigg)^{\frac{1}{p}}.
\end{align*}
Therefore, by inequality (\ref{equinorm}) and Lemma \ref{twochi},
\begin{align*}
\Norm{b-\EE^{t}_{N+1}(b)}_{p}
&\lesssim\sum_{k=N+1}^{\infty}2^{-\frac{k\Q}{p}}\Big(\sum_{R\in\D^{t}_{k}}(|R|^{-\frac{1}{2}}|\int_{R}b(y)h^{t}_{R}(y)dy|)^{p}\Big)^{\frac{1}{p}}\\
&\lesssim\sum_{j=1}^{2}\sum_{k=N+1}^{\infty}2^{-\frac{k\Q}{p}}\Big(\sum_{R\in\D^{t}_{k}}|\inner{[M_{b},T](f_{F_{j}^{R}}), e_{E_{j}^{R}}}|^{p}\Big)^{\frac{1}{p}}.
\end{align*}
But since $\{f_{F_{j}^{R}}\}_{R\in\D^{t}}$ and $\{e_{E_{j}^{R}}\}_{R\in\D^{t}}$ are nearly weak orthonormal sequence due to Lemma \ref{NWO}, we apply Lemma \ref{compact} to get 
$$ \sum_{j=1}^{2}\sum_{k=N+1}^{\infty}2^{-\frac{k\Q}{p}}\Big(\sum_{R\in\D^{t}_{k}}|\inner{[M_{b},T](f_{F_{j}^{R}}), e_{E_{j}^{R}}}|^{p}\Big)^{\frac{1}{p}} \lesssim2^{-\frac{(N+1)\Q}{p}}\Norm{[M_{b},T]}_{\L_{p}}.$$
Combining the above inequalities, we conclude the desired assertion.
\end{proof}

\begin{lemma}\label{core}
Keep the assumptions in Lemma \ref{tail}. We have
\begin{align*}
\Big(\sum_{k\in\ZZ}2^{k\Q}\Norm{\EE^{t}_{k+1}(b)-\EE^{t}_{k}(b)}_{p}^{p}\Big)^{\frac{1}{p}}\lesssim\Norm{[M_{b},T]}_{\L_{p}}.
\end{align*}
\end{lemma}
\begin{proof}
For infinite sum, it suffices to treat its arbitrary finite sum. Without loss of generality, fix a large positive integer $N$ and a negative integer $L$. Denote
\begin{align*}
F_{L,N}=\sum_{k=L}^{N}2^{k\Q}\Norm{\EE^{t}_{k+1}(b)-\EE^{t}_{k}(b)}_{p}^{p}.
\end{align*}
By inequality (\ref{equinorm}),
\begin{align*}
F_{L,N}=\sum_{k=L}^{N}\sum_{R\in\D^{t}_{k}}\fint_{R}|\EE^{t}_{k+1}(b)(y)-\EE^{t}_{k}(b)(y)|^{p}dy
\lesssim\sum_{k=L}^{N}\sum_{R\in\D^{t}_{k}}\Big(|R|^{-\frac{1}{2}}|\int_{R}b(y)h^{t}_{R}(y)dy|\Big)^{p}.
\end{align*}
Repeating the steps in the proof of Lemma \ref{tail}, we have
$$F_{L,N}   \lesssim   \sum_{k=L}^{N} \sum_{R\in\D^{t}_{k}}|\inner{[M_{b},T](f_{F_{j}^{R}}), e_{E_{j}^{R}}}|^{p} . $$
By Lemma \ref{compact} again,
\begin{align*}
F_{L,N}\lesssim\Norm{[M_{b},T]}_{\L_{p}} ^p.
\end{align*}
Here, the relevant constant does not depend on $L,N$. Letting $L\rightarrow-\infty$ and $N\rightarrow+\infty$ gives the desired result.
\end{proof}

\begin{lemma}\label{final}
Keep the assumptions in Lemma \ref{tail}. We have
\begin{align*}
\Big(\sum_{k\in\ZZ}2^{k\Q}\Norm{b-\EE^{t}_{k}(b)}_{p}^{p}\Big)^{\frac{1}{p}}\lesssim\Norm{[M_{b},T]}_{\L_{p}} .
\end{align*}
\end{lemma}
\begin{proof}
Fix a large positive integer $N$ and a negative integer $L$. Using the triangle inequality, we obtain
\begin{align*}
\Big(\sum_{k=L}^{N}2^{k\Q}\Norm{b-\EE^{t}_{k}(b)}_{p}^{p}\Big)^{\frac{1}{p}} &=  \Big(\sum_{k=L}^{N}2^{k\Q}\Norm{b-\EE^{t}_{k+1}(b)+\EE^{t}_{k+1}(b)-\EE^{t}_{k}(b)}_{p}^{p}\Big)^{\frac{1}{p}}   \\
&\leq2^{\frac{N\Q}{p}}\Norm{b-\EE^{t}_{N+1}(b)}_{p}
+2^{-\frac{\Q}{p}}\Big(\sum_{k=L}^{N}2^{k\Q}\Norm{b-\EE^{t}_{k}(b)}_{p}^{p}\Big)^{\frac{1}{p}}\\
&\qquad+\Big(\sum_{k=L}^{N}2^{k\Q}\Norm{\EE^{t}_{k+1}(b)-\EE^{t}_{k}(b)}_{p}^{p}\Big)^{\frac{1}{p}}.
\end{align*}
Note that $0<1-2^{-\frac{\Q}{p}}<1$ for $p\in(1,\infty)$. Therefore,
\begin{align*}
\Big(\sum_{k=L}^{N}2^{k\Q}\Norm{b-\EE^{t}_{k}(b)}_{p}^{p}\Big)^{\frac{1}{p}}
\lesssim2^{\frac{N\Q}{p}}\Norm{b-E_{N+1}(b)}_{p}+\Big(\sum_{k=L}^{N}2^{k\Q}\Norm{\EE^{t}_{k+1}(b)-\EE^{t}_{k}(b)}_{p}^{p}\Big)^{\frac{1}{p}}.
\end{align*}
Here the constant only depends on $p$ and $\Q$. Thus, Lemmas \ref{tail} and \ref{core} imply the desired result.
\end{proof}

\begin{proposition}\label{thm1}
Assume that $T$ is a Calder\'{o}n--Zygmund singular integral operator satisfying \eqref{not-change-sign} and \eqref{nondegenerate}, and that $b\in {\rm VMO}(\GG)$. If $p\in(\Q,\infty)$ and $[M_{b},T]\in\L_{p}$, then 
$$\Norm{b}_{B_{p}^{\frac{\Q}{p}}}\lesssim\Norm{[M_{b},T]}_{\L_{p}}.$$
\end{proposition}
\begin{proof}
Let $k_{0}>\log_{2}((1+\frac{1}{2}\gamma_{2})A_{0})$ be a fixed integer.
Note that $0<\frac{\Q}{p}<1$. By definition of $B_{p}^{\frac{\Q}{p}}$,
\begin{align}\label{besov-upper-bd}
\Norm{b}_{B_{p}^{\frac{\Q}{p}}}^{p}
\leq2^{2(k_{0}+4)\Q}\sum_{k\in\ZZ}2^{2k\Q}\int_{\GG}\int_{2^{-k-k_{0}-4}\leq\rho(y\rp x)\leq2^{-k-k_{0}-3}}|b(x)-b(y)|^{p}dydx.
\end{align}
Denote
\begin{align*}
J_{k+k_{0}+3}= \int_{\GG}\int_{\rho(y\rp x)\leq2^{-k-k_{0}-3}}|b(x)-b(y)|^{p}dydx.
\end{align*}
For $R'\in\D^{1}_{k+k_{0}+3}$, let 
$$U_{R'}=\set{y\in\GG:\inf_{z\in R'}\rho(y\rp z)<2^{-k-k_{0}-3}}.$$
Then
\begin{align}\label{Jkk3}
\nonumber J_{k+k_{0}+3}
&=\sum_{R'\in\D^{1}_{k+k_{0}+3}}\int_{R'}\int_{\rho(y\rp x)\leq2^{-k-k_{0}-3}}|b(x)-b(y)|^{p}dydx\\
&\leq\sum_{R'\in\D^{1}_{k+k_{0}+3}}\int_{R'}\int_{U_{R'}}|b(x)-b(y)|^{p}dydx.
\end{align}

By the property of dyadic cube, we have $\rho(z\rp c_{R'})\leq\gamma_{2}2^{-k-k_{0}-3}$ for $z\in R'$.
As $y\in U_{R'}$, by definition of infimum, select $z_{0}\in R'$ satisfying $\rho(y\rp z_{0})\leq2^{-k-k_{0}-2}$.
Then
\begin{align*}
\rho(y\rp c_{R'})\leq A_{0}\Big(\rho(y\rp z_{0})+\rho(z_{0}\rp c_{R'})\Big)
\leq(1+\frac{1}{2}\gamma_{2})A_{0}2^{-k-k_{0}-2}\leq2^{-k-2}.
\end{align*}
This implies that
$$U_{R'}\subset B(c_{R'},2^{-k-2}).$$
By the property of adjacent dyadic cube, there is some $t\in\{1,2,\ldots,\T\}$ and $R\in\D^{t}_{k}$ such that
\begin{align*}
U_{R'}\subset B(c_{R'},2^{-k-2})\subset R.
\end{align*}
Note that each $R\in\D^{t}_{k}$ contains at most $2^{(k_{0}+3)\Q}$ cubes $R'\in\D^{1}_{k+k_{0}+3}$. Therefore, by \eqref{Jkk3},
\begin{align*}
J_{k+k_{0}+3}\leq2^{(k_{0}+3)\Q}\sum_{t=1}^{\T}\sum_{R\in\D^{t}_{k}}\int_{R}\int_{R}|b(x)-b(y)|^{p}dydx.
\end{align*}

Writing 
$$|b(x)-b(y)| =  |b(x)-b_{R}+ b_{R}-b(y)|.$$
By inequality \eqref{besov-upper-bd}, the triangle inequality and Lemma \ref{final}, we obtain
\begin{align*}
\Norm{b}_{B_{p}^{\frac{\Q}{p}}}
\lesssim\Big(\sum_{k\in\ZZ}2^{2k\Q}J_{k+k_{0}+3}\Big)^{\frac{1}{p}}
&\lesssim\sum_{t=1}^{\T}\Big(\sum_{k\in\ZZ}2^{2k\Q}\sum_{R\in\D^{t}_{k}}\int_{R}\int_{R}|b(x)-b_{R}|^{p}dydx\Big)^{\frac{1}{p}}\\
&\lesssim\sum_{t=1}^{\T}\Big(\sum_{k\in\ZZ}2^{k\Q}\Norm{b-\EE^{t}_{k}(b)}_{p}^{p}\Big)^{\frac{1}{p}}\\
&\lesssim\Norm{[M_{b},T]}_{\L_{p}}.
\end{align*}
This is the desired result.
\end{proof}

\subsection{Lower bound}\label{mt-mta-lower}
Let $1\leq p,q<\infty$. The mixed norm space $L^{p}(L^{q,\infty})$ is defined as the set of the measurable function $G$ on $\GG\times\GG$ such that
\begin{align*}
\Norm{G}_{L^{p}(L^{q,\infty})}=\Bigg(\int_{\GG}\Norm{G(x,\cdot)}_{L^{q,\infty}(\GG)}^{p}dx\Bigg)^{\frac{1}{p}}<\infty.
\end{align*}
Let $G^{*}(x,y)=\overline{G(y,x)}$ and $p'$ be the conjugate number of $p$. 
It is shown in \cite{JW1982,Russo1977} that, if $G,G^{*}\in L^{p}(L^{p',\infty})$ with $p>2$, then $Af(x):=\int_{\GG}G(x,y)f(y)dy$ gives a compact operator $\L_{p,\infty}$ such that
\begin{align}\label{spsym}
\Norm{A}_{\L_{p,\infty}}\lesssim\max\{\Norm{G}_{L^{p}(L^{p',\infty})},\Norm{G^{*}}_{L^{p}(L^{p',\infty})}\}.
\end{align}

\begin{proposition}\label{thm2}
Assume that $T$ is a Calder\'{o}n--Zygmund singular integral operator. If $p\in(\Q,\infty)$ and $b\in B_{p}^{\frac{\Q}{p}}$, then 
$$\Norm{[M_{b},T]}_{\L_{p}}\lesssim\Norm{b}_{B_{p}^{\frac{\Q}{p}}}.$$
\end{proposition}
\begin{proof}
For $q>0$ and $x\in\GG$, we have
\begin{align*}
\Bigg|\set{y\in\GG: \rho(y^{-1}x)^{-\frac{\Q}{q}}>\lambda}\Bigg|=
\Bigg|\set{y\in\GG: \rho(y^{-1}x)^{\frac{\Q}{q}}<\frac{1}{\lambda}}\Bigg|=|B(o,1)|\lambda^{-q}.
\end{align*}

For $p>\Q\geq 4$, denote $p'$ the conjugate number of $p$ and $q=  (\frac{1}{p'}-\frac{1}{p} )^{-1}$, then
\begin{align*}
\sup_{x\in\GG}\Norm{\rho((\cdot)^{-1}x)^{\frac{\Q}{p}-\frac{Q}{p'}}}_{L^{q,\infty}}
=\sup_{x\in\GG}\sup_{\lambda>0}\lambda^{q}\Big|\set{y\in\GG: \rho(y^{-1}x)^{-\frac{\Q}{q}}>\lambda}\Big|=|B(o,1)|.
\end{align*}

By H\"{o}lder's inequality,
\begin{align*}
\int_{\GG}\Norm{(b(x)-b(\cdot))K(x,\cdot)}^{p}_{L^{p',\infty}}dx
&\lesssim \int_{\GG}\Norm{\frac{(b(x)-b(\cdot))}{\rho((\cdot)^{-1}x)^{\Q}}}^{p}_{L^{p',\infty}}dx\\
&\lesssim\int_{\GG}\Norm{\frac{(b(x)-b(\cdot))}{\rho((\cdot)^{-1}x)^{\frac{2\Q}{p}}}}^{p}_{L^{p}} \Norm{\rho((\cdot)^{-1}x)^{\frac{\Q}{p}-\frac{\Q}{p'}}}^{p}_{L^{q,\infty}}dx\\
&\lesssim\Norm{b}_{B_{p}^{\frac{\Q}{p}}}.
\end{align*}
Thus, by the symmetry of $(b(x)-b(y))K(x,y)$ and \eqref{spsym}, we have 
\begin{align*}
\Norm{[M_{b},T]}_{\L_{p,\infty}}\lesssim\Norm{b}_{B_{p}^{\frac{\Q}{p}}},\quad\forall p\in(\Q,\infty).
\end{align*}
Lastly, for $p\in(\Q,\infty)$, choose $p_{1},p_{2}\in(\Q,\infty)$ such that $\frac{1}{p}=\frac{\theta}{p_{1}}+\frac{1-\theta}{p_{2}}$ with $\theta\in(0,1)$. 
Then interpolation theorem gives the desired result (see e.g. \cite{Simon1979}).
\end{proof}

Propositions \ref{thm1} and \ref{thm2} complete the proof of Theorem \ref{mt} \eqref{mta}.
Checking the proof of \Cref{thm1} and \Cref{thm2}, we obtain the following characterisation of Besov space $B_{p}^{\frac{\Q}{p}}$ with $\Q<p<\infty$. 
\begin{corollary}
Let $\Q<p<\infty$ and $f\in L_{loc}(\GG)$. Then $f\in B_{p}^{\frac{\Q}{p}}$ if and only if $\displaystyle\sum_{k\in\ZZ}2^{k\Q}\Norm{b-\EE^{t}_{k}(b)}_{p}^{p}<\infty$ for all $t=1,\ldots,\T$.
Moreover, we have $$\Norm{f}_{B_{p}^{\frac{\Q}{p}}}\simeq
\sum_{t=1}^{\T}\Big(\sum_{k\in\ZZ}2^{k\Q}\Norm{b-\EE^{t}_{k}(b)}_{p}^{p}\Big)^{\frac{1}{p}}.$$
\end{corollary}

\section{Proof of \Cref{mt} \eqref{mtb}}\label{mt-mtb}

In this section, we provide the proof for the case that commutator $[M_{b},T]$ vanishes. This proof is based on the special structure of the multiplication mapping on $\GG$ that ensures a lower estimate.
According to stratification $\RR^{d}=\RR^{n_{1}}\times\RR^{n_{2}}\times\cdots\times\RR^{n_{\tau}}$,  it is natural to write $y\rp x=((y\rp x)^{(1)},(y\rp x)^{(2)},\ldots,(y\rp x)^{(\tau)})$ for $x,y\in\GG$. 
More precisely, the components of $(y^{-1}x)^{(1)}$ can be written as follows
\begin{align}\label{inverse-firstlayer}
(y^{-1}x)^{(1)}_{l}=x^{(1)}_{l}-y^{(1)}_{l},~l=1,2,\ldots,n_{1},
\end{align}
where $x^{(1)}_{l}$ is the $l$-th component in the first layer of $x$ and $y^{(1)}_{l}$ is the $l$-th component in the first layer of $y$.

Let $\set{X_{1}^{(j)},\ldots,X_{n_{j}}^{(j)}}$ be a linear basis of $\g_{j}$ for $j=1,2,\ldots,\tau$, and denote $\set{X_{1}^{(1)},\ldots,X_{n_{1}}^{(1)}}=\set{X_{1},\ldots,X_{n_{1}}}$ for simplicity.
The first order Mac Laurin formula \cite[Prop~20.3.11]{BLU2007} or Taylor expansion \cite[Prop~20.3.14]{BLU2007} on $\GG$ with integral remainder can be described as follows. If $u\in C^{2}(\GG)$, then 
\begin{align}\label{Taylor-formula}
u(x)=u(y)+\sum_{j=1}^{n_{1}}X_{j}u(y)(y^{-1}x)^{(1)}_{j}+\Omega(u,y\rp x),
\end{align}
where
\begin{align*}
\Omega(u,y\rp x)&=\sum_{j=2}^{\tau}\sum_{k=1}^{n_{j}}X_{k}^{(j)}u(y)(y\rp x)_{k}^{(j)}\\
&+\sum_{j_{1},j_{2}\in\{1,\ldots,\tau\}}\sum_{\substack{k_{1}\in\{1,\ldots,n_{j_{1}}\}\\k_{2}\in\{1,\ldots,n_{j_{2}}\}}}
\frac{(y\rp x)_{k_{1}}^{(j_{1})}(y\rp x)_{k_{2}}^{(j_{2})}}{2}\int_{0}^{1}X_{k_{1}}^{(j_{1})}X_{k_{2}}^{(j_{2})}u\Big(y \Exp(\sum_{j=1}^{\tau}\sum_{k=1}^{n_{j}}s(y\rp x)_{k}^{(j)}X_{k}^{(j)})\Big)(1-s)ds.
\end{align*}
Letting $j\in\{2,\ldots,\tau\}$ and $k\in\{1,\ldots,n_{j}\}$, we have $|(y\rp x)_{k}^{(j)}|\leq\rho(y\rp x)^{j}$ and $X_{k}^{(j)}=[X_{i_{1}},[\cdots,[X_{i_{j-1}},X_{i_{j}}]]]$ for some $i_{1},\ldots,i_{j}\in\set{1,\ldots,n_{1}}$.
Thus, for $y^{-1}x$ near $o$, we have
\begin{align}\label{Remainder}
|\Omega(u,y\rp x)|\leq C_{\GG} \; \rho(y\rp x)^{2}\sup_{\substack{2\leq m\leq2\tau\\k_{1},\ldots,k_{m}\in\{1,\ldots,n_{1}\}}}\Norm{X_{k_{1}}\cdots X_{k_{m}}u}_{\infty}
\end{align}
for some positive constant $C_{\GG}$ depending only on $\GG$.

%



\begin{lemma}\label{firstlayer}
Let $k\in\ZZ$, and $C_{\rho}, A_{0}$ be given in \eqref{norm-inequality} and \eqref{triangle-inequality} respectively. There are an integer $N_{0}>\log_{2}(2A_{0}C_{\rho})+1$ and a constant $C_{1}>0$ such that, for any $R\in\D_{k}$ and $s\in\{-1,1\}^{n_{1}}$, we can find $R_{1},R_{2}\in\D_{k+N_{0}}$ satisfying that $R_{1},R_{2}\subset R$ and
\begin{align*}
s_{j}(x^{(1)}_{j}-y^{(1)}_{j})\geq C_{1}2^{-k},\quad j=1,2,\ldots,n_{1},  \quad \forall \;x\in R_{1}, \forall \;y\in R_{2},
\end{align*}
where $s_{j}$ is the $j$-th component of $s$.
\end{lemma}
\begin{proof}
Denote $v(w)$ the direction of non-zero vector $w\in\RR^{n_{1}}$.
For cube $R\in\D_{k}$, by the properties of dyadic cube system, we find positive constants $\gamma_1 $ and $\gamma_2$ such that $B(c_{R},\gamma_{1}2^{-k})\subset R \subset B(c_{R},\gamma_{2}2^{-k})$. Write $R=c_{R}\cdot R_{o}$ with $o$ the center of $R_{o}$. By translation, we have
$$B(o,\gamma_{1}2^{-k})\subset R_{o}\subset B(o,\gamma_{2}2^{-k}).$$

For $s\in\{-1,1\}^{n_{1}}$, select $x_{R}\in B(o,\gamma_{1}2^{-k})$ satisfying
\begin{align*}
v(x_{R}^{(1)})=v(s),\quad \rho_{\infty}(x_{R})=\frac{1}{2A_{0}C_{\rho}}\gamma_{1}2^{-k}
\quad{\rm and}\quad |(x_{R}^{(1)})_{j}|=\frac{1}{2A_{0}C_{\rho}}\gamma_{1}2^{-k},j=1,2,\ldots,n_{1}.
\end{align*}
If $\rho(z)=\gamma_{1}2^{-k}$, by the symmetry of $\rho$, we have
\begin{align*}
\rho(z\rp x_{R})=\rho(x_{R}\rp z)\geq \frac{1}{A_{0}}\rho(z)-\rho(x_{R})\geq\frac{1}{A_{0}}\rho(z)-C_{\rho}\rho_{\infty}(x_{R})=\frac{1}{2A_{0}}\gamma_{1}2^{-k}.
\end{align*}
This implies that the distance between $x_{R}$ and the sphere of the ball $B(o,\gamma_{1}2^{-k})$ is larger than $\frac{1}{2A_{0}}\gamma_{1}2^{-k}$.
Combining the fact that $\rho(x_R)\geq \rho_{\infty}(x_{R})=\frac{1}{2A_{0}C_{\rho}}\gamma_{1}2^{-k}$, we may choose large enough $N>\log_{2}(2A_{0}C_{\rho})+1$ such that
\begin{align*}
o\notin B(x_{R},\gamma_{1}2^{-k-N})\subset B(o,\gamma_{1}2^{-k}).
\end{align*}
By the definition of dilation, the size in the first layer of $\GG$ does not change. So, for $x\in B(x_{R},\gamma_{1}2^{-k-N})$, we have
\begin{align*}
\frac{1}{4A_{0}C_{\rho}}\gamma_{1}2^{-k}\leq
(\frac{1}{2A_{0}C_{\rho}}-2^{-N})\gamma_{1}2^{-k}\leq|x^{(1)}_{j}|
\leq(\frac{1}{2A_{0}C_{\rho}}+2^{-N})\gamma_{1}2^{-k}
\leq\frac{1}{A_{0}C_{\rho}}\gamma_{1}2^{-k}
\end{align*}
for any $j=1,2,\ldots,n_{1}$.

By the symmetry of $\rho$,
\begin{align*}
B(x_{R},\gamma_{1}2^{-k-N})\rp=B(y_{R},\gamma_{1}2^{-k-N})
\end{align*}
with $y_{R}=x_{R}\rp$. By formula \eqref{inverse-firstlayer}, $v(y_{R}^{(1)})=-v(s)$, and
\begin{align*}
\frac{1}{4A_{0}C_{\rho}}\gamma_{1}2^{-k}\leq |y^{(1)}_{j}| 
\leq\frac{1}{A_{0}C_{\rho}}\gamma_{1}2^{-k},\quad j=1,2,\ldots,n_{1} \quad \forall \; y\in B(y_{R},\gamma_{1}2^{-k-N}).
\end{align*}

In a sum, for $N$ chosen as above, there is a constant $C_{1}>0$ satisfying that
\begin{align*}
s_{j}(x^{(1)}_{j}-y^{(1)}_{j})\geq C_{1}2^{-k},\quad j=1,2,\ldots,n_{1}  
\end{align*}
for any pair of $x\in B(x_{R},\gamma_{1}2^{-k-N})$ and $  y\in B(y_{R},\gamma_{1}2^{-k-N})$.
Moreover, by formula \eqref{inverse-firstlayer}, the same estimate holds true for any pair of $x\in c_{R}\cdot B(x_{R},\gamma_{1}2^{-k-N})$ and $  y\in c_{R}\cdot B(y_{R},\gamma_{1}2^{-k-N})$.
Lastly, pick a large integer $N_{0}\geq N>\log_{2}(2A_{0}C_{\rho})+1$ such that we can find cubes $R_{1},R_{2}\in\D_{k+N_{0}}$ with 
\begin{align*}
R_{1}\subset c_{R}\cdot B(x_{R},\gamma_{1}2^{-k-N})\subset R\quad \mbox{and}\quad 
R_{2}\subset c_{R}\cdot B(y_{R},\gamma_{1}2^{-k-N})\subset R.
\end{align*}
This is the expected result.
\end{proof}

Now fix $N_0$ as in the above lemma.
For $R\in\D_{k}$, denote
\begin{align*}
\E(b,R)=\fint_{R}\fint_{R}|\EE_{k+N_{0}}(b)(x)-\EE_{k+N_{0}}(b)(y)|dxdy.
\end{align*}

\begin{lemma}\label{constantcontra}
Assume $b\in C_{c}^{\infty}(\GG)$ and $\set{\E(b,R)}_{R\in\D}\in\ell^{\Q}$. Then $b$ is a constant.
\end{lemma}
\begin{proof}
Recall $\nabla=(X_{1},X_{2},\ldots,X_{n_{1}})$.
If $b\in C_{c}^{\infty}(\GG)$ is not a constant, then there is at least one point $x_{0}\in\GG$ such that $|\nabla(b)(x_{0})|>0$.
For $x,y\in\GG$, and $R \in \D_k $ with center $c_{R}$, by formula \eqref{inverse-firstlayer}, we have
\begin{align*}
(c_{R}^{-1}x)^{(1)}-(c_{R}^{-1}y)^{(1)}=x^{(1)}-y^{(1)}=(y^{-1}x)^{(1)}.
\end{align*}
By the Taylor expansion formula \eqref{Taylor-formula},
\begin{align*}
b(x)-b(y)&=(b(x)-b(c_{R}))-(b(y)-b(c_{R}))\\
&=\sum_{j=1}^{n_{1}}X_{j}b(c_{R})(y^{-1}x)^{(1)}_{j}+\Omega(b,c_{R}\rp x)-\Omega(b,c_{R}\rp y).
\end{align*}
When $x\in R$, it follows from \eqref{Remainder} that 
\begin{align*}
 |\Omega(b,c_{R}\rp x)| \leq C_{b}2^{-2k}
\end{align*}
for some positive constant $C_{b}$ independent of $k$. In order to estimate $X_{j}b(c_{R})(y^{-1}x)^{(1)}_{j}$, denote $s^{R}_{j}=\sgn(X_{j}b(c_{R}))$, for $j=1,2,\ldots,n_{1}$. For $s^R:= (s^{R}_{1},\ldots, s^{R}_{n_1})$, by Lemma \ref{firstlayer}, we find $R_{1},R_{2} \subset R$ such that, for $x\in R_{1},y\in R_{2}$, we have
\begin{align*}
b(x)-b(y)&\geq\sum_{j=1}^{n_{1}}|X_{j}b(c_{R})|s^{R}_{j}(y^{-1}x)^{(1)}_{j}-2C_{b}2^{-2k}\\
&\geq C_{1}2^{-k}|\nabla b(c_{R})|-2C_{b}2^{-2k}.
\end{align*}

For $b\in C_{c}^{\infty}(\GG)$, there is an $\epsilon_{0}>0$ satisfying that, when $\rho(y\rp x_{0})\leq\epsilon_{0}$, then
$|\nabla b(y)|\geq\frac{1}{2}|\nabla b(x_{0})|$.
Let $N_{\epsilon_{0}}$ be the least integer such that there exists cube $R\in\D_{N_{\epsilon_{0}}}$ satisfying that $\rho(c_R\rp x_{0})\leq\epsilon_{0}$.

Continue the estimate of $b(x)-b(y)$.
Take $N'_{0}>\max(\log_{2}(\frac{8C_{b}}{C_{1}|\nabla b(x_{0})|}),N_{\epsilon_{0}})$, and let $k\geq N'_{0}$. Then for $R\in\D_{k}$ with $\rho(c_{R}\rp x_{0})\leq\epsilon_{0}$, we have
\begin{align*}
b(x)-b(y)\geq\frac{1}{2}C_{1}2^{-k}|\nabla b(x_{0})|-2C_{b}2^{-2k}
\geq\frac{1}{4}C_{1}2^{-k}|\nabla b(x_{0})|\quad x\in R_{1},y\in R_{2}.
\end{align*}
It then follows that
\begin{align*}
|b_{R_{1}}-b_{R_{2}}|=\bigg|\fint_{R_{1}}\fint_{R_{2}}(b(x)-b(y))dxdy\bigg|
\geq\frac{1}{4}C_{1}2^{-k}|\nabla b(x_{0})|.
\end{align*}
When $R_{1},R_{2}\in\D_{k+N_{0}}$ are the descendants of $R$, by definition,
$$\E(b,R)\geq 2^{-2N_0}\fint_{R_1}\fint_{R_2}|\EE_{k+N_{0}}(b)(x)-\EE_{k+N_{0}}(b)(y)|dxdy= 2^{-2N_0}|b_{R_{1}}-b_{R_{2}}|.$$
Thus,
\begin{align*}
	\Norm{\set{\E(b,R)}_{R\in\D}}_{\ell^{\Q}}^{\Q}
	&\geq\sum_{k>N'_{0}}\sum_{\substack{ R\in\D_{k} \\ \rho(c_{R}\rp x_{0})\leq\epsilon_{0}}}  2^{-2N_0 Q}|b_{R_{1}}-b_{R_{2}}|^{\Q}\\
	&\geq\bigg(\frac{1}{4}C_{1}|\nabla b(x_{0})|\bigg)^{\Q}\sum_{k>N'_{0}}2^{-k\Q}\sum_{R\in\D_{k},\rho(c_{R}\rp x_{0})\leq\epsilon_{0}}\unit.
\end{align*}
Note that, for large enough $k>N'_{0}$, there are at least $2^{(k-N_{\epsilon_{0}})\Q}$ cubes $R\in\D_{k}$ such that $\rho(c_{R}\rp x_{0})\leq\epsilon_{0}$. So the last sum above is divergent. This is a contradiction. We complete the proof.
\end{proof}

\begin{proposition}\label{constantthm}
Assume that $T$ is a Calder\'{o}n--Zygmund singular integral operator satisfying \eqref{not-change-sign} and \eqref{nondegenerate}, $b\in {\rm VMO}(\GG)$ and $[M_{b},T]\in\L_{\Q}$. Then $b$ is a constant.
\end{proposition}
\begin{proof}
Note that $C_{c}^{\infty}(\GG)$ is dense in ${\rm VMO}(\GG)$ (see e.g. \cite{CDLW2019}). By identity approximation in \cite[Prop~(1.20)]{FS1982}, it suffices to suppose $b\in C_{c}^{\infty}(\GG)$. We are going to show that 
\begin{align*}
\Norm{\set{\E(b,R)}_{R\in\D}}_{\ell^{\Q}}
\lesssim\Norm{[M_{b},T]}_{\L_{\Q}}.
\end{align*}
Indeed, for $R\in \D_k$, writing
$$\EE_{k+N_{0}}(b)(x)-\EE_{k+N_{0}}(b)(y)=\EE_{k+N_{0}}(b)(x)-b_{R}+b_{R}-\EE_{k+N_{0}}(b)(y),$$
it follows from the triangle inequality that
\begin{align*}
\E(b,R)
\leq2\fint_{R}|\EE_{k+N_{0}}(b)(x)-b_{R}|dx
=2 \fint_{R}|\EE_{k+N_{0}}(b)(x)-\EE_{k}(b)(x)|dx.
\end{align*}
By Lemma \ref{core},
\begin{align*}
\Norm{\set{\E(b,R)}_{R\in\D}}_{\ell^{\Q}}
\lesssim\Big(\sum_{k\in\ZZ}2^{k\Q}\Norm{\EE_{k+N_{0}}(b)-\EE_{k}(b)}_{\Q}^{\Q}\Big)^{\frac{1}{\Q}}
\lesssim\Norm{[M_{b},T]}_{\L_{\Q}}.
\end{align*}
Then Lemma \ref{constantcontra} immediately implies that $b$ is a constant.
\end{proof}

\text{Proof of Theorem \ref{mt} \eqref{mtb} :}
\begin{proof}
When $0<p\leq \Q$, by H\"{o}lder's inequality,
\begin{align*}
\Norm{[M_{b},T]}_{\L_{\Q}}\leq\Norm{[M_{b},T]}_{\L_{\infty}}^{1-\frac{p}{\Q}}\Norm{[M_{b},T]}_{\L_{p}}^{\frac{p}{\Q}}.
\end{align*}
Then the desired assertion follows immediately from Proposition \ref{constantthm}.
\end{proof}

\section{Proof of \Cref{endponit-estimates}}\label{endponit-estimates-sect}
The proof of \Cref{endponit-estimates} is divided into two parts in which the proof of sufficiency is arranged in the first part and the proof of necessity is arranged in the second part.

\subsection{Alpert bases}
The terminology of Alpert bases is first introduced in \cite{A1993}, and then adapted to positive Borel measures on $\RR^{n}$ in \cite{RSW2021}. For our purpose, let us construct the Alpert bases on $\GG$.
Recall that $\GG=(\RR^{d},\circ)$ with homogeneous dimension $\Q$.
For $\alpha\in\NN^{d}$, write $\alpha$ in the form related to the stratification of $\GG$ by $(\alpha^{(1)},\alpha^{(2)},\ldots,\alpha^{(\tau)})$. We use the notation
\begin{align*}
|\alpha|_{\GG}=\sum_{j=1}^{\tau}\sum_{k=1}^{n_{j}}j\cdot\alpha_{k}^{(j)}.
\end{align*}
For $\alpha\in\NN^{d}$ and $x\in\RR^{d}$, denote $x^{\alpha}=x_{1}^{\alpha_{1}}x_{2}^{\alpha_{2}}\cdots x_{d}^{\alpha_{d}}$.
For $Q\in\D$, number the children of $Q$ by $Q_{1},\ldots,Q_{2^{\Q}}$ and set
\begin{align*}
|Q_{k}^{(\beta)}|=\int_{Q_{k}}x^{\beta}dx,\quad \beta\in\NN^{d},\quad 1\leq k\leq2^{\Q}.
\end{align*}
Let $\displaystyle h(x)=\sum_{k=1}^{2^{\Q}}\sum_{|\alpha|_{\GG}\leq m-1} c_{\alpha,k}x^{\alpha}\unit_{Q_{k}}$ with $c_{\alpha,k}\in\RR$
and we consider (cancellation conditions) equations
\begin{align*}
\int_{Q}h(x)x^{\beta}dx=0\quad\mbox{for}\quad\mbox{all}\quad |\beta|_{\GG}\leq m-1.
\end{align*}
Denote $N_{d,m}$ by the cardinality of set $\set{\alpha\in\NN^{d}:|\alpha|_{\GG}\leq m-1}$ and label its elements as $\alpha_{(1)},\ldots,\alpha_{(N_{d,m})}$. Then the above equations can be interpreted in the following way
\begin{align*}
\sum_{k=1}^{2^{\Q}}
\begin{pmatrix}
  |Q_{k}^{(\alpha_{(1)}+\beta_{(1)})}| & |Q_{k}^{(\alpha_{(1)}+\beta_{(2)})}| & \cdots & |Q_{k}^{(\alpha_{(1)}+\beta_{(N_{d,m})})}| \\
  |Q_{k}^{(\alpha_{(2)}+\beta_{(1)})}| & |Q_{k}^{(\alpha_{(2)}+\beta_{(2)})}| & \cdots & |Q_{k}^{(\alpha_{(2)}+\beta_{(N_{d,m})})}| \\
  \vdots & \vdots & \ddots & \vdots \\
  |Q_{k}^{(\alpha_{(N_{d,m})}+\beta_{(1)})}| & |Q_{k}^{(\alpha_{(N_{d,m})}+\beta_{(2)})}| & \cdots & |Q_{k}^{(\alpha_{(N_{d,m})}+\beta_{(N_{d,m})})}| 
\end{pmatrix}
\begin{pmatrix}
  c_{\alpha^{(1)},k} \\
  c_{\alpha^{(2)},k} \\
  \vdots \\
  c_{\alpha^{(N_{d,m})},k}
\end{pmatrix}
=\begin{pmatrix}
  0 \\
  0 \\
  \vdots \\
  0
\end{pmatrix}.
\end{align*}
The number of variables is strictly greater than the number of equations, so this linear equation system has non-zero solutions. This implies that cancellation conditions are available.
For $Q\in\D$, let $child(Q)$ be the set containing the $2^{\Q}$ children of $Q$. For $m\in\NN$, denote
\begin{align*}
L^{2}_{m}(Q)={\rm Span}
\set{h(x)=\sum_{Q'\in child(Q)}\unit_{Q'}p_{Q',m}(x):\int_{Q}h(x)x^{\beta}dx=0\quad \mbox{for}\quad \beta\in\NN^{d}\quad\mbox{with}\quad|\beta|_{\GG}\leq m-1},
\end{align*}
where $\displaystyle p_{Q',m}(x)=\sum_{\alpha\in\NN^{d},|\alpha|_{\GG}\leq m-1}a_{Q',\alpha}x^{\alpha}$ with $a_{Q',\alpha}\in\RR$ is a polynomial on $\GG$. A polynomial of the form $\displaystyle\sum_{Q'\in child(Q)}\unit_{Q'}p_{Q',m}(x)$ is called an Alpert function.

Let $Q_{0}$ be the translation of cubes in $\D_{0}$ with respect to their centers. Given $m\in\NN$, a polynomial $\displaystyle P(x)=\sum_{\alpha\in\NN^{d},|\alpha|_{\GG}\leq m-1}c_{\alpha}x^{\alpha}$ of homogeneous degree not greater than $m-1$ is said to be $Q_{0}$-normalized if $\displaystyle\sup_{x\in Q_{0}}|P(x)|=1$. 
For any polynomial $P$, we say it is $Q$-normalized if $P^{Q}$ is $Q_{0}$-normalized where $P^{Q}(x)=P\left(c_{Q}\delta_{\frac{1}{r_{Q}}}(x)\right)$ ($c_{Q}$ is the center of $Q$ and $r_{Q}$ is its radius).
The following lemma is analogous to \cite[Lemma 20]{Sawyer2022}.
\begin{lemma}\label{Q-normalized lemma}
For $m\in\NN$, there is a constant $C_{m,\GG}$ satisfying that, for each  $Q\in\D$ and all $Q$-normalized polynomial on $\GG$ of homogeneous degree not greater than $m$,  we have
\begin{align*}
|Q|\leq C_{m,\GG}\int_{Q}|P(x)|^{2}dx.
\end{align*}
\end{lemma}

If $\int_{Q}x^{\beta}dx\neq0$ for all $|\beta|_{\GG}\leq m-1$, we can decompose $L^{2}_{m}(Q)$ as the following way
\begin{align}\label{L2Q-decomposition}
{\rm Span}\set{\unit_{Q'}x^{\alpha}:Q'\in child(Q),\quad|\alpha|_{\GG}\leq m-1}
\ominus\left(\bigoplus_{\alpha:|\alpha|_{\GG}\leq m-1}\unit_{Q}x^{\alpha}\right).
\end{align}
Let $\Delta_{Q}^{m}$ be the orthogonal projection onto $L^{2}_{m}(Q)$ and let $\EE_{Q}^{m}$ be the orthogonal projection onto the finite dimensional subspace
\begin{align*}
{\rm Span}\set{\unit_{Q}(x)x^{\alpha}:0\leq|\alpha|_{\GG}\leq m-1}.
\end{align*}
We next investigate the existence of the Alpert basis on $\lp{2}$ and its properties.
\begin{theorem}\label{Alpert-bases}
Given fixed $m\in\NN$, let $\D$ be a dyadic cube system on $\GG$ with $\int_{Q}x^{\beta}dx\neq0$ for all $Q\in\D$ and $|\beta|_{\GG}\leq m-1$.
There exists an orthonomal basis   $\set{h_{Q,m}^{j}}_{j\in\Gamma_{\Q}^{m},Q\in\D}$ of $\lp{2}$, which is called Alpert basis of order m, satisfying the following properties:
\begin{enumerate}[\rm(a)]
\item\label{orthoganal basis} $\set{h_{Q,m}^{j}}_{j\in\Gamma_{\Q}^{m}}$ is the orthonomal basis of $L^{2}_{m}(Q)$ with  $\supp(h_{Q,m}^{j})\subset Q$ and $\inner{h_{P,m}^{j},h_{Q,m}^{k}}=0$ for $P\neq Q\in\D$, $j,k\in\Gamma_{\Q}^{m}$.
Here $\Gamma_{\Q}^{m}$ is an index set with cardinality not greater than $(2^{\Q}-1)\binom{d+m-1}{d}$.

\item\label{moment conditions} The moment conditions:
\begin{align*}
\int_{Q}h_{Q,m}^{j}(x)x^{\alpha}dx=0\quad\mbox{for}\quad |\alpha|_{\GG}\leq m-1,\quad j\in\Gamma_{\Q}^{m},Q\in\D.
\end{align*} 

\item\label{size conditions} Size estimates:
\begin{align}\label{expect size conditions}
\Norm{\EE_{Q}^{m}(f)}_{\lp{\infty}}\lesssim\fint_{Q}|f(\xi)|d\xi,\quad Q\in\D,f\in L_{loc}(\GG),
\end{align}
and
\begin{align}\label{Haar size conditions}
\Norm{h_{Q,m}^{j}}_{\lp{\infty}}\simeq|Q|^{-\frac{1}{2}}\quad \mbox{for}\quad Q\in\D.
\end{align}

\item\label{telescoping identities} The telescoping identities:
\begin{align*}
\unit_{P}\sum_{Q\subset R\subset P}\sum_{j\in\Gamma_{\Q}^{m}}\inner{f,h_{R,m}^{j}}h_{R,m}^{j}=\EE_{Q}^{m}(f)-\EE_{P}^{m}(f)
\quad\mbox{for}\quad f\in\lp{2}, P,Q\in\D.
\end{align*}

\item\label{pointwise almost everywhere} We have  $\displaystyle f=\sum_{Q\in\D}\sum_{j\in\Gamma_{\Q}^{m}}\inner{f,h_{Q,m}^{j}}h_{Q,m}^{j}$ both in $\lp{2}$-norm and pointwise almost everywhere.

\end{enumerate}
\end{theorem}

The following proofs are induced from \cite{RSW2021} and \cite{Sawyer2022}.
\begin{proof}
For \eqref{orthoganal basis}, 
note that $L^{2}_{m}(Q)$  is a finite dimensional linear space for $Q\in\D$. By theory in Hilbert space, there exists an orthonomal basis $\set{h_{Q,m}^{j}}_{j\in\Gamma_{\Q}^{m}}$ of $L^{2}_{m}(Q)$ consisting of Alpert functions. 
Clearly, $h_{Q,m}^{j}$ is supported in $Q$.
Note that, for each $Q\in\D$, the dimension of space ${\rm Span}\set{\unit_{Q}(x)x^{\alpha}:0\leq|\alpha|_{\GG}\leq m-1}$ is not greater than $\displaystyle\sum_{j=1}^{m-1}\binom{d-1+j}{j}=\binom{d-1+m}{m-1}$. Since there are $2^{\Q}$ children in $child(Q)$ for each $Q\in\D$, it follows from \eqref{L2Q-decomposition} that
\begin{align*}
{\rm dim}(L^{2}_{m}(Q))\leq(2^{\Q}-1)\binom{d-1+m}{m-1}=(2^{\Q}-1)\binom{d+m-1}{d}.
\end{align*}
This is the estimate for the cardinality of $\Gamma_{\Q}^{m}$.

Let us claim that $L^{2}_{m}(P)$ is orthogonal to $L^{2}_{m}(Q)$ for $P\neq Q\in\D$. It is trivial if $P$ or $Q$ does not contain each other. If $P\subset Q$, then the restriction of $h\in L^{2}_{m}(Q)$ to $P$ is a polynomial with homogeneous degree not greater than $m-1$. By definition, $L^{2}_{m}(P)$ is orthogonal to such a function. So is the similar reason for the case $Q\subset P$.
Therefore, $h_{P,m}^{j}$ is orthogonal to $h_{Q,m}^{k}$ for $P\neq Q\in\D$, $j,k\in\Gamma_{\Q}^{m}$.

For \eqref{moment conditions}, combining $h_{Q,m}^{j}\in L^{2}_{m}(Q)$ and the definition of $L^{2}_{m}(Q)$, we obtain the vanishing moments.

For \eqref{size conditions}, since $\frac{h_{Q,m}^{j}}{\Norm{h_{Q,m}^{j}}_{L^{\infty}(Q)}}$ is $Q$-normalized, it follows from Lemma \ref{Q-normalized lemma} and $\Norm{h_{Q,m}^{j}}_{\lp{2}}=1$ that
\begin{align*}
|Q|\simeq\Norm{h_{Q,m}^{j}}_{L^{\infty}(Q)}^{-2}\int_{Q}|h_{Q,m}^{j}(x)|^{2}dx
=\Norm{h_{Q,m}^{j}}_{\lp{\infty}}^{-2}.
\end{align*}
This is \eqref{Haar size conditions}. Similar assertions in \cite[4.2.2]{Sawyer2022} give \eqref{expect size conditions}.

For \eqref{telescoping identities}, since $\Delta_{Q}^{m}$ is the orthogonal projection onto $L^{2}_{m}(Q)$ and $\set{h_{Q,m}^{j}}_{j\in\Gamma_{\Q}^{m}}$ is the orthonomal basis of $L^{2}_{m}(Q)$, we have $\displaystyle\Delta_{Q}^{m}(f)=\sum_{j\in\Gamma_{\Q}^{m}}\inner{f,h_{Q,m}^{j}}h_{Q,m}^{j}$. Then, by \eqref{L2Q-decomposition}, $\displaystyle\EE_{Q}^{m}(f)-\EE_{P}^{m}(f)=\unit_{P}\sum_{Q\subset R\subset P}\Delta_{R}^{m}(f)$  gives the identities.

For \eqref{pointwise almost everywhere}, since $L^{2}_{m}(R)$ is contained in $\lp{2}$ for $R\in\D$, we see that the direct sum of $\set{L^{2}_{m}(R)}_{R\in\D}$ is contained in $\lp{2}$, i.e.
\begin{align*}
\bigoplus_{R\in\D} L^{2}_{m}(R)\subset\lp{2}.
\end{align*}
Conversely, by \cite{KLPW2016}, $\lp{2}$ is a direct sum of $\set{L^{2}_{1}(Q)}_{Q\in\D}$ because these spaces are generated by Haar bases. Noting that, for $Q\in\D$, $$L^{2}_{1}(Q)\subset \overline{{\rm Span}}\set{L^{2}_{m}(R):R\in\D,R\supset Q}.$$
we conclude 
\begin{align*}
\lp{2}=\bigoplus_{Q\in\D}L^{2}_{1}(Q)\subset\bigoplus_{R\in\D} L^{2}_{m}(R).
\end{align*}
This is the desired result. 
\end{proof}

Denote $N_{d,m}=\binom{d-1+m}{m-1}$ and let $\set{\omega_{Q,m}^{j}}_{j=1}^{N_{d,m}}$ of the form $\displaystyle \sum_{\alpha\in\NN^{d},|\alpha|_{\GG}\leq m-1}c_{\alpha}x^{\alpha}\unit_{Q}$ be an orthonormal basis for vector space ${\rm Span}\set{\unit_{Q}(x)x^{\alpha}:0\leq|\alpha|_{\GG}\leq m-1}$ with respect to the inner product of $\lp{2}$. It is obvious that $\supp(\omega_{Q,m}^{j})\subset Q$. By a similar argument to \eqref{Haar size conditions} (see also \cite[4.2.2]{Sawyer2022}), we obtain
\begin{align}\label{basis-EQm}
\Norm{\omega_{Q,m}^{j}}_{\infty}\simeq|Q|^{-\frac{1}{2}}.
\end{align}
Since $\EE_{Q}^{m}$ is the orthogonal projection onto the finite dimensional subspace
${\rm Span}\set{\unit_{Q}(x)x^{\alpha}:0\leq|\alpha|_{\GG}\leq m-1}$, we are able to give an explicit formula of $\EE_{Q}^{m}$ as following
\begin{align}\label{EQm}
\EE_{Q}^{m}=\sum_{j=1}^{N_{d,m}}\inner{\cdot,\omega_{Q,m}^{j}}\omega_{Q,m}^{j},\quad Q\in\D.
\end{align}

Furthermore, by \eqref{telescoping identities} and \eqref{pointwise almost everywhere} in Theorem \ref{Alpert-bases},
\begin{align}\label{Haar-expansion}
f\unit_{R}-\EE_{R}^{m}(f)=\sum_{\substack{Q\in\D\\Q\subset R}}\Delta_{Q}^{m}(f)=\sum_{\substack{Q\in\D\\Q\subset R}}\sum_{j\in\Gamma_{\Q}^{m}}\inner{f,h_{Q,m}^{j}}h_{Q,m}^{j}\quad\mbox{for}\quad R\in\D, ~f\in \lp{2}.
\end{align}
For later use, let us denote the local difference between  $f\unit_{R}$ and $\EE_{R}^{m}(f)$ by
\begin{align*}
\HH_{R}^{m}(f)(x)=\sum_{\substack{Q\in\D\\Q\subset R}}\sum_{j\in\Gamma_{\Q}^{m}}\inner{f,h_{Q,m}^{j}}h_{Q,m}^{j}(x),\quad f\in\lp{p},~m\in\NN.
\end{align*}

\subsection{Proof of sufficiency in \Cref{endponit-estimates}}\label{endponit-estimates-sect-su}
In this part, we are going to show
\begin{align*}
\Norm{[M_{b},T]}_{\L_{\Q,\infty}}\lesssim \Norm{b}_{\swq}.
\end{align*}
We will apply the Alpert bases to locally expand the kernel $K$. But before that, we need to decompose the kernel $K$ on the cubes from the dyadic Whitney decomposition on $\GG$. 
So let us state the dyadic Whitney decomposition on $\GG$ at first.

Let $\D$ be a dyadic cube system on $\GG$ and $\D_{k}$ be the corresponding k-th dyadic cube system.
Denote the product dyadic cube system on $\GG\times\GG$ by 
$$\D_{prod}=\set{\D_{k}\times\D_{k}}_{k\in\ZZ}.$$
Here, the product set $\D_{k}\times\D_{k}$ is interpreted as the set containing all cubes $Q_{1}\times Q_{2}$ for $Q_{1},Q_{2}\in\D_{k}$.
Using the similar argument as in Euclidean space \cite[Appendix J]{G2014-c} (also \cite[page 168]{Stein1970}), one has the following dyadic Whitney decomposition.
\begin{lemma}\label{Whitney-decomposition}
Suppose that $\Omega$ is an open nonempty proper subset of $\GG\times\GG$. Then there is a countable family $\P$ of dyadic cubes in $\D_{prod}$ such that
\begin{enumerate}[\rm(1)]
\item The cubes in $\P$ have disjoint interiors and $\displaystyle\Omega=\bigcup_{P\in\P}P$;

\item For $P\in\P$, the sidelength of $P$ is comparable to the distance from $P$ to $\Omega^{c}$; 

\item If the boundaries of $P,Q\in\P$ touch, then their sidelengths are comparable; 

\item There are some $\varepsilon>0$ and a relevant positive constant $N_{\P,\varepsilon}$ satisfying that $\displaystyle\sum_{P\in\P}\unit_{(1+\varepsilon)P}(x)\leq N_{\P,\varepsilon}$ for all $x\in\Omega$. 
\end{enumerate}
\end{lemma}
%
%
%

\begin{proof}[Proof of sufficiency in \Cref{endponit-estimates}]
Let $\D$ be given as in Theorem \ref{Alpert-bases}. 
Let $$\Omega=\GG\times\GG\setminus\set{(x,y)\in\GG\times\GG:x= y}$$ and let $\P$ be its related family of dyadic Whitney decomposition in \Cref{Whitney-decomposition}, that is $\displaystyle\Omega=\bigcup_{P\in\P}P$. 
Therefore, we write $\displaystyle K(x,y)=\sum_{P\in\P}K(x,y)\unit_{P}(x,y)$ with $P=P_{1}\times P_{2}$ where $P_{1},P_{2}\in\D$ have the same sidelength and the distance from $P_{1}$ to $P_{2}$ is comparable to their sidelength.
Thus, for each $R\in\D$, there are at most $N_{\P}$ cubes $V_{R,s}$ related to $R$ such that $R\times V_{R,s}\in\P$.
So we can reorganize the sum
\begin{align*}
K(x,y)=\sum_{P\in\P}K(x,y)\unit_{P}(x,y)=\sum_{R\in\D}\sum_{s=1}^{N_{\P}}K(x,y)\unit_{R\times V_{R,s}}(x,y),
\end{align*}
where $R$ and $V_{R,s}$ has the same sidelength and the distance between them is comparable to their sidelength. 
Now, we apply \eqref{Haar-expansion} to write
\begin{align*}
\unit_{R}(x)K(x,y)\unit_{V_{R,s}}(y)=
(\EE_{R}^{k_{1}}+\HH_{R}^{k_{1}})\otimes(\EE_{V_{R,s}}^{k_{2}}+\HH_{V_{R,s}}^{k_{2}})\left(K(\unit_{R}\otimes\unit_{V_{R,s}})\right)(x,y)
\end{align*}
in the sense of $L^{2}(\GG\times\GG)$ with $\Q<k_{1},k_{2}\leq \gamma_{T}$.
More precisely, we locally decompose the kernel $K$ as the following four parts
\begin{align}\label{kernel-decomposition}
\unit_{R}(x)K(x,y)\unit_{V_{R,s}}(y)
=F_{1,s,R}(x,y)+F_{2,s,R}(x,y)+F_{3,s,R}(x,y)+F_{4,s,R}(x,y),
\end{align}
where
\begin{align*}
&F_{1,s,R}(x,y)=
\EE_{R}^{k_{1}}\otimes\EE_{V_{R,s}}^{k_{2}}\left(K(\unit_{R}\otimes\unit_{V_{R,s}})\right)(x,y),\\
&F_{2,s,R}(x,y)=
\sum_{\substack{Q\in\D\\Q\subset R}}\sum_{j\in\Gamma_{\Q}^{k_{1}}}\int_{Q}\EE_{V_{R,s}}^{k_{2}}\left(K(\xi,\cdot)\right)(y) h_{Q,k_{1}}^{j}(\xi)d\xi h_{Q,k_{1}}^{j}(x) \unit_{V_{R,s}}(y),\\
&F_{3,s,R}(x,y)=
\sum_{\substack{\tilde Q\in\D\\\tilde{Q}\subset V_{R,s}}}\sum_{\wh{j}\in\Gamma_{\Q}^{k_{2}}}\int_{\wh{Q}}\EE_{R}^{k_{1}}\left(K(\cdot,\xi)\right)(x) h_{\wh{Q},k_{2}}^{\wh{j}}(\xi)d\xi h_{\wh{Q},k_{2}}^{\wh{j}}(y) \unit_{R}(x),\\
&F_{4,s,R}(x,y)=
\sum_{\substack{Q\in\D\\Q\subset R}}\sum_{j\in\Gamma_{\Q}^{k_{1}}}\sum_{\substack{\tilde Q\in\D\\\tilde{Q}\subset V_{R,s}}}\sum_{\wh{j}\in\Gamma_{\Q}^{k_{2}}} \lambda^{R}_{Q,j,\tilde Q,\tilde j}h_{Q,k_{1}}^{j}(x)h_{\wh{Q},k_{2}}^{\tilde j}(y).
\end{align*}
Here,
$$\lambda^{R}_{Q,j,\tilde Q,\tilde j}= \int_{\GG}\int_{\GG}\unit_{R}(x)K(x,y)\unit_{V_{R,s}}(y) h_{Q,k_{1}}^{j}(x)h_{\wh{Q},k_{2}}^{\tilde j}(y) dxdy.$$
We next show that the integrals associated to $F_{1,s,R}(x,y)$, $F_{2,s,R}(x,y)$, $F_{3,s,R}(x,y)$ and $F_{4,s,R}(x,y)$ are dominated by the mean oscillation of function $b$.

For $F_{1,s,R}$, write
\begin{align*}
F_{1,s,R}(x,y)
=|R|\EE_{R}^{k_{1}}\otimes\EE_{V_{R,s}}^{k_{2}}\left(K(\unit_{R}\otimes\unit_{V_{R,s}})\right)(x,y) \frac{\unit_{R}(x)}{|R|^{\frac{1}{2}}}\frac{\unit_{V_{R,s}}(y)}{|V_{R,s}|^{\frac{1}{2}}}
=: C_{R,s}(x,y)\frac{\unit_{R}(x)}{|R|^{\frac{1}{2}}}\frac{\unit_{V_{R,s}}(y)}{|V_{R,s}|^{\frac{1}{2}}}.
\end{align*}
By \eqref{expect size conditions} and $|K(\unit_{R}\otimes\unit_{V_{R,s}})(x,y)|\lesssim\frac{1}{|R|}$, one has
\begin{align*}
|C_{R,s}(x,y)|\lesssim|R|\fint_{R}\fint_{V_{R,s}}|K(\xi,\eta)|d\xi d\eta\lesssim1.
\end{align*}
Noting that there is a large positive constant $\gamma>1$ such that $\gamma R$ contains $R\bigcup V_{R,s}$,
we introduce the following notion
\begin{align}\label{osc-gamma}
\osc_{\gamma}(b,R)=\Big(\frac{1}{|R|}\int_{\gamma R}|b(u)-b_{R}|^{3}du\Big)^{\frac{1}{3}},
\end{align}
and denote
\begin{align*}
\Phi_{\gamma,R}(x)=\frac{(b(x)-b_{R})\unit_{R}(x)}{\osc_{\gamma}(b,R)|R|^{\frac{1}{2}}}
\quad\mbox{and}\quad
\wh{\Phi}_{\gamma,V_{R,s}}(y)=\frac{(b_{R}-b(y))\unit_{V_{R,s}}(y)}{\osc_{\gamma}(b,R)|V_{R,s}|^{\frac{1}{2}}}.
\end{align*}
For each cube $R$, rewrite $b(x)-b(y)=(b(x)-b_{R})+(b_{R}-b(y))$. We obtain
\begin{align*}
&\sum_{R\in\D}C_{R,s}(x,y) (b(x)-b(y))\frac{\unit_{R}(x)}{|R|^{\frac{1}{2}}}
\frac{\unit_{V_{R,s}}(y)}{|V_{R,s}|^{\frac{1}{2}}}\\
=&\sum_{R\in\D}\osc_{\gamma}(b,R)C_{R,s}(x,y) \Phi_{\gamma,R}(x)\frac{\unit_{V_{R,s}}(y)}{|V_{R,s}|^{\frac{1}{2}}}
+\sum_{R\in\D}\osc_{\gamma}(b,R)C_{R,s}(x,y) \frac{\unit_{R}(x)}{|R|^{\frac{1}{2}}}\wh{\Phi}_{\gamma,V_{R,s}}(y).
\end{align*}
Thus,
\begin{equation}\label{commutator-F1sR}
\begin{split}
\sum_{R\in\D}\int_{\GG}(b(x)-b(y))F_{1,s,R}(x,y)f(y)dy
&=\sum_{R\in\D}\int_{\GG}C_{R,s}(x,y)(b(x)-b(y))
\frac{\unit_{R}(x)}{|R|^{\frac{1}{2}}}\frac{\unit_{V_{R,s}}(y)}{|V_{R,s}|^{\frac{1}{2}}}f(y)dy\\
&=\sum_{R\in\D}\osc_{\gamma}(b,R) \Phi_{\gamma,R}(x)
\int_{\GG}C_{R,s}(x,y)\frac{\unit_{V_{R,s}}(y)}{|V_{R,s}|^{\frac{1}{2}}}f(y)dy\\
&+\sum_{R\in\D}\osc_{\gamma}(b,R) \frac{\unit_{R}(x)}{|R|^{\frac{1}{2}}}
\int_{\GG}C_{R,s}(x,y)\wh{\Phi}_{\gamma,V_{R,s}}(y)f(y)dy.
\end{split}
\end{equation}
It is obvious that $\supp(\Phi_{\gamma,R})$ is contained in $R$ and
$\supp(\wh{\Phi}_{\gamma,V_{R,s}})$ is contained in $V_{R,s}$. 
Simple calculations yield that 
\begin{align*}
\Norm{\Phi_{\gamma,R}}_{\lp{3}}
\lesssim|R|^{-\frac{1}{6}}
\quad\mbox{and}\quad
\Norm{\wh{\Phi}_{\gamma,V_{R,s}}}_{\lp{3}}
\lesssim|V_{R,s}|^{-\frac{1}{6}}.
\end{align*}
By \Cref{NWO}, the sequences $\set{|\Phi_{\gamma,R}|}_{R\in\D}$ and $\set{|\wh{\Phi}_{\gamma,V_{R,s}}|}_{R\in\D}$ are NWO sequences for index 3. 
Therefore, \eqref{commutator-F1sR} and \Cref{compact-upper-bd-function} imply that
\begin{align}\label{F-1-s-R}
\Norm{\sum_{R\in\D}[M_{b},F_{1,s,R}]}_{\L_{\Q,\infty}}
\lesssim\Norm{\set{\osc_{\gamma}(b,R)}_{R\in\D}}_{\ell^{\Q,\infty}},
\end{align}
where $F_{1,s,R}$ represents the integral operator with $F_{1,s,R}(x,y)$ as its integral kernel.

For $F_{2,s,R}$, using \eqref{Haar-expansion}, we write
\begin{align*}
&\sum_{j\in\Gamma_{\Q}^{k_{1}}}\sum_{\substack{Q\in\D\\Q\subset R}}\int_{Q}\EE_{V_{R,s}}^{k_{2}}\left(K(\xi,\cdot)\right)(y) h_{Q,k_{1}}^{j}(\xi)d\xi h_{Q,k_{1}}^{j}(x)\\
=&\EE_{V_{R,s}}^{k_{2}}\left(K(x,\cdot)\right)(y)\unit_{R}(x)
-\EE_{R}^{k_{1}}\otimes\EE_{V_{R,s}}^{k_{2}}\left(K(\unit_{R}\otimes\unit_{V_{R,s}})\right)(x,y).
\end{align*}
Therefore,
\begin{align*}
&\sum_{R\in\D}\int_{\GG}(b(x)-b(y))F_{2,s,R}(x,y)f(y)dy\\
=&\sum_{R\in\D}\int_{\GG}(b(x)-b(y))\EE_{V_{R,s}}^{k_{2}}\left(K(x,\cdot)\right)(y)\unit_{R}(x)f(y)dy\\
-&\sum_{R\in\D}\int_{\GG}(b(x)-b(y))
\EE_{R}^{k_{1}}\otimes\EE_{V_{R,s}}^{k_{2}}\left(K(\unit_{R}\otimes\unit_{V_{R,s}})\right)(x,y)f(y)dy.
\end{align*}
In the above last line, the $\L_{\QQ,\infty}$ estimate of the second factor is due to \eqref{F-1-s-R} and it is left to deal with the first factor. 
By \eqref{expect size conditions} and $|K(\unit_{R}\otimes\unit_{V_{R,s}})(x,y)|\lesssim\frac{1}{|R|}$, we obtain
\begin{align*}
|R||\EE_{V_{R,s}}^{k_{2}}\left(K(x,\cdot)\right)(y)\unit_{R}(x)|
\lesssim|R|\fint_{V_{R,s}}|K(x,\eta)|d\eta\unit_{R}(x)
\lesssim1.
\end{align*}
Rewrite
\begin{align*}
&\sum_{R\in\D}\int_{\GG}(b(x)-b(y))\EE_{V_{R,s}}^{k_{2}}\left(K(x,\cdot)\right)(y)\unit_{R}(x)f(y)dy\\
=&\sum_{R\in\D}\int_{\GG}|R|\EE_{V_{R,s}}^{k_{2}}\left(K(x,\cdot)\right)(y)(b(x)-b(y))
\frac{\unit_{R}(x)}{|R|^{\frac{1}{2}}} \frac{\unit_{V_{R,s}}(y)}{|R|^{\frac{1}{2}}}f(y)dy.
\end{align*}
So the $\L_{\QQ,\infty}$ estimate of this term is similar to $F_{1,s,R}$.
In a sum, we have
\begin{align}\label{F-2-s-R}
\Norm{\sum_{R\in\D}[M_{b},F_{2,s,R}]}_{\L_{\Q,\infty}}
\lesssim\Norm{\set{\osc_{\gamma}(b,R)}_{R\in\D}}_{\ell^{\Q,\infty}}
\end{align}
where $F_{2,s,R}$ represents the integral operator with $F_{2,s,R}(x,y)$ as its integral kernel.
By a similar estimate of $F_{2,s,R}$ for $F_{3,s,R}$, we have
\begin{align}\label{F-3-s-R}
\Norm{\sum_{R\in\D}[M_{b},F_{3,s,R}]}_{\L_{\Q,\infty}}
\lesssim\Norm{\set{\osc_{\gamma}(b,R)}_{R\in\D}}_{\ell^{\Q,\infty}},
\end{align}
where $F_{3,s,R}$ represents the integral operator with $F_{3,s,R}(x,y)$ as its integral kernel.

For $F_{4,s,R}$, using \eqref{osc-gamma} again, we denote
\begin{align*}
H_{R,Q,j}(x)=\frac{b(x)-b_{R}}{\osc_{\gamma}(b,R)}\left(\frac{|Q|}{|R|}\right)^{\frac{1}{2}}h_{Q,k_{1}}^{j}(x)\unit_{R}(x),
\quad
\wh{H}_{R,\wh{Q},\wh{j},s}(y)=\left(\frac{|\wh{Q}|}{|R|}\right)^{\frac{1}{2}}h_{\wh{Q},k_{2}}^{\wh{j}}(y)\unit_{V_{R,s}}(y),
\end{align*}
and
\begin{align*}
G_{R,Q,j}(x)=\left(\frac{|Q|}{|R|}\right)^{\frac{1}{2}}h_{Q,k_{1}}^{j}(x)\unit_{R}(x),
\quad
\wh{G}_{R,\wh{Q},\wh{j},s}(y)=\frac{b_{R}-b(y)}{\osc_{\gamma}(b,R)}
\left(\frac{|\wh{Q}|}{|R|}\right)^{\frac{1}{2}}h_{\wh{Q},k_{2}}^{\wh{j}}(y)\unit_{V_{R,s}}(y).
\end{align*}
It is obvious that $\set{\wh{H}_{R,\wh{Q},\wh{j},s}}_{R\in\D}$ and $\set{G_{R,Q,j}}_{R\in\D}$ are NWO sequences. Simple computations yield that
\begin{align*}
\supp(H_{R,Q,j})\subset R,\quad
\Norm{H_{R,Q,j}}_{\lp{3}}\lesssim|R|^{-\frac{1}{6}}
\end{align*}
and
\begin{align*}
\supp(\wh{G}_{R,\wh{Q},\wh{j},s})\subset V_{R,s},\quad
\Norm{\wh{G}_{R,\wh{Q},\wh{j},s}}_{\lp{3}}\lesssim|V_{R,s}|^{-\frac{1}{6}}.
\end{align*}
By \Cref{NWO}, $\set{H_{R,Q,j}}_{R\in\D}$ and $\set{\wh{G}_{R,\wh{Q},\wh{j},s}}_{R\in\D}$ are NWO sequences for index 3.  
Writing $b(x)-b(y)$ as $(b(x)-b_{R})+(b_{R}-b(y))$, we deduce
\begin{equation}\label{convergent-decompose}
\begin{split}
&\sum_{R\in\D}\int_{\GG}(b(x)-b(y))F_{4,s,R}(x,y)f(y)dy\\
=&\sum_{j\in\Gamma_{\Q}^{k_{1}},\wh{j}\in\Gamma_{\Q}^{k_{2}}}\sum_{R\in\D}\osc_{\gamma}(b,R) \sum_{\substack{Q\in\D\\Q\subset R}}
\sum_{\substack{\tilde Q\in\D\\\tilde{Q}\subset V_{R,s}}}
\lambda^{R}_{Q,j,\tilde Q,\tilde j}\left(\frac{|R|}{|Q|}\right)^{\frac{1}{2}} \left(\frac{|R|}{|\wh{Q}|}\right)^{\frac{1}{2}} H_{R,Q,j}(x) \inner{f,\wh{H}_{R,\wh{Q},\wh{j},s}}\\
+&\sum_{j\in\Gamma_{\Q}^{k_{1}},\wh{j}\in\Gamma_{\Q}^{k_{2}}}\sum_{R\in\D}\osc_{\gamma}(b,R) \sum_{\substack{Q\in\D\\Q\subset R}}
\sum_{\substack{\tilde Q\in\D\\\tilde{Q}\subset V_{R,s}}}
\lambda^{R}_{Q,j,\tilde Q,\tilde j}\left(\frac{|R|}{|Q|}\right)^{\frac{1}{2}} \left(\frac{|R|}{|\wh{Q}|}\right)^{\frac{1}{2}} G_{R,Q,j}(x) \inner{f,\wh{G}_{R,\wh{Q},\wh{j},s}}.
\end{split}
\end{equation}
To analyse the decay in the last two summations on the right hand side,  we only have to consider the factor in the first one, since the two summations have the same decay rate. Let $R$ be in the $l_{R}$-th dyadic cube system. Assume that $Q\in\D_{l_{1}}$ and $\tilde{Q}\in\D_{l_{2}}$ with $l_{1},l_{2}\geq l_{R}$ satisfying that $Q\subset R$ and $\tilde{Q}\subset V_{R,s}$. 
Applying the left Taylor expansion formula \cite[Chapter 20.3.3]{BLU2007} (or \cite[(1.44) Corollary]{FS1982}) to $K(\xi,\eta)$ with respect to $\xi$ at point $c_{Q}$ and then to $\eta$ at point $c_{\wh{Q}}$ yields
\begin{align*}
K(\xi,\eta)
&=P_{Q}^{(k_{1})}(\xi,\eta)+P_{Q}^{(k_{1},k_{2})}(\xi,\eta)+R_{Q}^{(k_{1},k_{2})}(\xi,\eta),
\end{align*}
where $P_{Q}^{(k_{1})}(\xi,\eta)$ is the product of a polynomial of $\xi$ of homogeneous degree not greater than $k_{1}-1$ and a continuous function of $\eta$, and $P_{Q}^{(k_{1},k_{2})}(\xi,\eta)$ is the product of a polynomial of $\xi$ (of homogeneous degree equal $k_{1}$) and a polynomial of $\eta$ (of homogeneous degree not greater than $k_{2}-1$), and the remainder $R_{Q}^{(k_{1},k_{2})}(\xi,\eta)$ satisfies the estimate
\begin{align*}
|R_{Q}^{(k_{1},k_{2})}(\xi,\eta)|\lesssim
\frac{1}{\rho(c_{Q},c_{\wh{Q}})^{\Q+k_{1}+k_{2}}}\quad\mbox{for}\quad \xi\in Q,\eta\in\wh{Q}.
\end{align*}
Therefore, by \eqref{moment conditions} and \eqref{Haar size conditions} in Theorem \ref{Alpert-bases}, we have
\begin{align*}
\lambda^{R}_{Q,j,\tilde Q,\tilde j}
&=\int_{Q}\int_{\wh{Q}}\left(K(\xi,\eta)-P_{Q}^{(k_{1})}(\xi,\eta)-P_{Q}^{(k_{1},k_{2})}(\xi,\eta)\right)h_{Q,k_{1}}^{j}(\xi)h_{\wh{Q},k_{2}}^{\wh{j}}(\eta)d\xi d\eta\\
&=\int_{Q}\int_{\wh{Q}}R_{Q}^{(k_{1},k_{2})}(\xi,\eta)h_{Q,k_{1}}^{j}(\xi)h_{\wh{Q},k_{2}}^{\wh{j}}(\eta)d\xi d\eta\\
&\lesssim\int_{Q}\int_{\wh{Q}}\frac{\rho(\xi,c_{Q})^{k_{1}}\rho(\eta,c_{\wh{Q}})^{k_{2}}}{\rho(c_{Q},c_{\wh{Q}})^{\Q+k_{1}+k_{2}}}
|Q|^{-\frac{1}{2}}|\wh{Q}|^{-\frac{1}{2}}d\xi d\eta\\
&\lesssim\frac{2^{-l_{1}k_{1}}2^{-l_{2}k_{2}}|Q|^{\frac{1}{2}}|\wh{Q}|^{\frac{1}{2}}}{2^{-l_{R}(\Q+k_{1}+k_{2})}}.
\end{align*}
Hence,
\begin{align}\label{convergent factor lambda}
\lambda^{R}_{Q,j,\tilde Q,\tilde j}\left(\frac{|R|}{|Q|}\right)^{\frac{1}{2}} \left(\frac{|R|}{|\wh{Q}|}\right)^{\frac{1}{2}}
\lesssim2^{(l_{R}-l_{1})k_{1}}2^{(l_{R}-l_{2})k_{2}}.
\end{align}
Write
\begin{equation}\label{convergent-factor}
\begin{split}
&\sum_{\substack{Q\in\D\\Q\subset R}}\sum_{\substack{\tilde Q\in\D\\\tilde{Q}\subset V_{R,s}}}
\lambda^{R}_{Q,j,\tilde Q,\tilde j}\left(\frac{|R|}{|Q|}\right)^{\frac{1}{2}} \left(\frac{|R|}{|\wh{Q}|}\right)^{\frac{1}{2}} H_{R,Q,j}(x)\wh{H}_{R,\wh{Q},\wh{j},s}(y)\\
&=\sum_{l_{1},l_{2}=l_{R}}^{+\infty}\frac{1}{2^{(l_{1}-l_{R})\Q}}\sum_{\substack{Q\in\D_{l_{1}}\\Q\subset R}}\frac{1}{2^{(l_{2}-l_{R})\Q}}\sum_{\substack{\tilde Q\in\D_{l_{2}}\\\tilde{Q}\subset V_{R,s}}}
2^{(l_{1}-l_{R})\Q}2^{(l_{2}-l_{R})\Q}\lambda^{R}_{Q,j,\tilde Q,\tilde j}\left(\frac{|R|}{|Q|}\right)^{\frac{1}{2}} \left(\frac{|R|}{|\wh{Q}|}\right)^{\frac{1}{2}} H_{R,Q,j}(x)\wh{H}_{R,\wh{Q},\wh{j},s}(y).
\end{split}
\end{equation}
Then \eqref{convergent factor lambda} guarantees
\begin{align*}
2^{(l_{1}-l_{R})\Q}2^{(l_{2}-l_{R})\Q}\lambda^{R}_{Q,j,\tilde Q,\tilde j}\left(\frac{|R|}{|Q|}\right)^{\frac{1}{2}} \left(\frac{|R|}{|\wh{Q}|}\right)^{\frac{1}{2}}
\lesssim2^{(l_{R}-l_{1})(k_{1}-\QQ)}2^{(l_{R}-l_{2})(k_{2}-\QQ)}.
\end{align*}
This provides a convergent factor for $\displaystyle\sum_{l_{1},l_{2}=l_{R}}^{+\infty}$ with $k_{1},k_{2}>\Q$.
Combining \eqref{convergent-decompose}, \eqref{convergent-factor} and  \Cref{compact-upper-bd-sequence}, we obtain
\begin{align}\label{F-4-s-R}
\Norm{\sum_{R\in\D}[M_{b},F_{4,s,R}]}_{\L_{\Q,\infty}}
\lesssim\Norm{\set{\osc_{\gamma}(b,R)}_{R\in\D}}_{\ell^{\Q,\infty}},
\end{align}
where $F_{4,s,R}$ is the integral operator with $F_{4,s,R}(x,y)$ as its integral kernel.

Lastly, since
\begin{align*}
[M_{b},T]f(x)
&=\sum_{R\in\D}\sum_{s=1}^{N_{\P}}\int_{\GG}(b(x)-b(y))K(x,y)\unit_{R\times V_{R,s}}(x,y)f(y)dy\\
&=\sum_{s=1}^{N_{\P}}\sum_{R\in\D}\int_{\GG}(b(x)-b(y))\left(F_{1,s,R}(x,y)+F_{2,s,R}(x,y)+F_{3,s,R}(x,y)+F_{4,s,R}(x,y) \right)f(y)dy,
\end{align*}
it follows from estimates \eqref{F-1-s-R}, \eqref{F-2-s-R}, \eqref{F-3-s-R} and \eqref{F-4-s-R} that
\begin{align*}
\Norm{[M_{b},T]}_{\L_{\Q,\infty}}
\lesssim\Norm{\set{\osc_{\gamma}(b,R)}_{R\in\D}}_{\ell^{\Q,\infty}}.
\end{align*}
Then \Cref{sobolev-osc} in the Appendix implies the desired result. 
The proof of sufficiency is completed.
\end{proof}

\subsection{Proof of necessity in \Cref{endponit-estimates}}\label{endponit-estimates-sect-ne}
In this part, we are going to show
\begin{align*}
\Norm{b}_{\swq}\lesssim \Norm{[M_{b},T]}_{\L_{\Q,\infty}}.
\end{align*}
Opposite to the previous part,  we will apply Alpert basis to locally expand $K\rp$ on the cubes using the non-degenerate conditions. 

\begin{proof}[Proof of necessity in \Cref{endponit-estimates}]
Let $\D$ be given as in Theorem \ref{Alpert-bases} and $\D_{k}$ be the corresponding k-th dyadic cube system.
By \Cref{cube-nondegenerate}, there are positive constants $A'_{4}\geq A'_{3}\geq2A_{0}$ such that, for each cube $R\in\D_{k}$ (with center $c_{R}$), one can find another cube $\wh{R}\in\D_{k}$ (with center $c_{\wh{R}}$) with  $A'_{3}2^{-k}\leq\rho(c_{R}\rp c_{\wh{R}})\leq A'_{4}2^{-k}$ satisfying that, $K(\wh{x},x)$ does not change sign and $|K(\wh{x},x)|\gtrsim|R|\rp$  for all $(\wh{x},x)\in \wh{R}\times R$.
Note also that by the size estimate of $K(\wh{x},x)$, we have $|K(\wh{x},x)|\lesssim|R|\rp$  for all $(\wh{x},x)\in \wh{R}\times R$.
Thus, we have $|K(\wh{x},x)|\approx|R|\rp$  for all $(\wh{x},x)\in \wh{R}\times R$, which gives
$$ |K(\wh{x},x)|^{-1}\approx|R|\qquad \forall (\wh{x},x)\in \wh{R}\times R. $$

Denote
\begin{align*}
J_{R}(\wh{x},x)=|R|^{-2}\unit_{\wh{R}}(\wh{x})K(\wh{x},x)\rp\unit_{R}(x).
\end{align*}
This allows us to write
\begin{align}\label{KJ-expansion}
K(\wh{x},x)J_{R}(\wh{x},x)
=\frac{\unit_{\wh{R}}(\wh{x})\unit_{R}(x)}{|\wh{R}|\;|R|}.
\end{align}
Let
\begin{align*}
\eta_{R}(y)=\sgn(b_{\wh{R}}-b(y))\unit_{R}(y)
\quad\mbox{and}\quad
L_{R}(f)(y)=\eta_{R}(y)\int_{\GG}J_{R}(w,y)f(w)dw.
\end{align*}
Simple computations yield that
\begin{align*}
[M_{b},T]L_{R}(f)(x)=
\int_{\wh{R}}\int_{R}(b(x)-b(y))\eta_{R}(y)K(x,y)J_{R}(w,y)dyf(w)dw.
\end{align*}
Therefore,
\begin{align*}
\Tr\Big([M_{b},T]L_{R}\Big)=
\int_{\wh{R}}\int_{R}(b(\wh{x})-b(x))\eta_{R}(x)K(\wh{x},x)J_{R}(\wh{x},x)dxd\wh{x}.
\end{align*}
Since $\displaystyle\fint_{\wh{R}}(b(\wh{x})-b_{\wh{R}})d\wh{x}=0$, it follows from \eqref{KJ-expansion} that
\begin{align*}
\bigg|\Tr\Big([M_{b},T]L_{R}\Big)\bigg|
=\bigg|\fint_{\wh{R}}\fint_{R}(b(\wh{x})-b(x))\eta_{R}(x)dxd\wh{x}\bigg|
=\fint_{R}|b_{\wh{R}}-b(x)|dx.
\end{align*}
And also,
\begin{equation}\label{sobolev-estimate}
\begin{split}
\fint_{R}|b(x)-b_{R}|dx
&\leq\fint_{R}|b(x)-b_{\wh{R}}|dx+|b_{R}-b_{\wh{R}}|\\
&\leq2\fint_{R}|b(x)-b_{\wh{R}}|dx\\
&\leq 2\bigg|\Tr\Big([M_{b},T]L_{R}\Big)\bigg|.
\end{split}
\end{equation}

We are reduced to estimating the $\ell^{\Q,\infty}$-norm of $\big|\Tr\Big([M_{b},T]L_{R}\Big)\big|$.
Recall the duality of Lorentz space in \cite[Theorem 1.4.16]{G2014-c} that $(\ell^{p,1})^{*}=\ell^{\frac{p}{p-1},\infty}$ for $1<p<\infty$. 
It follows that
\begin{equation}\label{duality-trace}
\begin{split}
\Norm{\set{\Tr\Big([M_{b},T]L_{R}\Big)}_{R\in\D}}_{\ell^{\Q,\infty}}
&=\sup_{\Norm{\{a_{R}\}_{R\in\D}}_{\ell^{\frac{\Q}{\Q-1},1}}\leq1}|
\Tr\Big(\sum_{R\in\D}[M_{b},T]L_{R}a_{R}\Big)|\\
&\leq\Norm{[M_{b},T]}_{\L_{\Q,\infty}}\sup_{\Norm{\{a_{R}\}_{R\in\D}}_{\ell^{\frac{\Q}{\Q-1},1}}\leq1}
\bigg\|{\sum_{R\in\D}L_{R}a_{R}}\bigg\|_{\L_{\frac{\Q}{\Q-1},1}} . 
\end{split}
\end{equation}
We now apply \eqref{Haar-expansion} to write
\begin{align*}
\unit_{\wh{R}}(\wh{x})K(\wh{x},x)\rp\unit_{R}(x)=
(\EE_{\wh{R}}^{k_{2}}+\HH_{\wh{R}}^{k_{2}})\otimes(\EE_{R}^{k_{1}}+\HH_{R}^{k_{1}})\left(K\rp(\unit_{\wh{R}}\otimes\unit_{R})\right)(\wh{x},x)
\end{align*}
in the sense of $L^{2}(\GG\times\GG)$ with $\Q<k_{1},k_{2}\leq \gamma_{T}$.
Thus,
\begin{align*}
&\sum_{R\in\D}a_{R}L_{R}(f)(x)\\
=&\sum_{R\in\D}a_{R}\frac{\eta_{R}(x)}{|R|^{2}}
\int_{\GG}\left(\EE_{\wh{R}}^{k_{2}}\otimes\EE_{R}^{k_{1}}+\EE_{\wh{R}}^{k_{2}}\otimes\HH_{R}^{k_{1}}
+\HH_{\wh{R}}^{k_{2}}\otimes\EE_{R}^{k_{1}}+\HH_{\wh{R}}^{k_{2}}\otimes\HH_{R}^{k_{1}}\right)
\left(K\rp(\unit_{\wh{R}}\otimes\unit_{R})\right)(\wh{x},x)f(\wh{x})d\wh{x}\\
=:&\;\Pi_{1}(f)(x)+\Pi_{2}(f)(x)+\Pi_{3}(f)(x)+\Pi_{4}(f)(x).
\end{align*}

For $\Pi_{1}(f)$, note that by \eqref{EQm}
\begin{align*}
\EE_{\wh{R}}^{k_{2}}\otimes\EE_{R}^{k_{1}}\left(K\rp(\unit_{\wh{R}}\otimes\unit_{R})\right)(\wh{x},x)
=\sum_{j_{2}=1}^{N_{d,k_{2}}}\sum_{j_{1}=1}^{N_{d,k_{1}}}
\inner{K\rp(\unit_{\wh{R}}\otimes\unit_{R}),\omega_{\wh{R},k_{2}}^{j_{2}}\otimes\omega_{R,k_{1}}^{j_{1}}}
\omega_{\wh{R},k_{2}}^{j_{2}}(\wh{x})\omega_{R,k_{1}}^{j_{1}}(x).
\end{align*}
 Let
\begin{align*}
\zeta_{1,R,j_{1},j_{2}}=\frac{1}{|R|^{2}}
\inner{K\rp(\unit_{\wh{R}}\otimes\unit_{R}),\omega_{\wh{R},k_{2}}^{j_{2}}\otimes\omega_{R,k_{1}}^{j_{1}}}.
\end{align*}
By \eqref{basis-EQm} and $|K\rp(\unit_{\wh{R}}\otimes\unit_{R})(x,y)|\lesssim|R|$, we have $|\zeta_{1,R,j_{1},j_{2}}|\lesssim1$.
Write $\Pi_{1}$ as the following form
\begin{align*}
\Pi_{1}(f)(x)=
\sum_{j_{2}=1}^{N_{d,k_{2}}}\sum_{j_{1}=1}^{N_{d,k_{1}}}\sum_{R\in\D}a_{R}\zeta_{1,R,j_{1},j_{2}} \omega_{R,k_{1}}^{j_{1}}(x)\eta_{R}(x)
\int_{\GG}\omega_{\wh{R},k_{2}}^{j_{2}}(\wh{x}) f(\wh{x})d\wh{x}.
\end{align*}
Clearly, $\set{\omega_{R,k_{1}}^{j_{1}}\eta_{R}}_{R\in\D}$ and $\set{\omega_{\wh{R},k_{2}}^{j_{2}}}_{R\in\D}$ are NWO sequences due to \eqref{basis-EQm}. 
By \Cref{compact-upper-bd}, we obtain
\begin{align}\label{Gamma-1}
\Norm{\Pi_{1}}_{\L_{\frac{\Q}{\Q-1},1}}
\lesssim \sum_{j_{2}=1}^{N_{d,k_{2}}}\sum_{j_{1}=1}^{N_{d,k_{1}}}
\Norm{\{a_{R}\zeta_{1,R,j_{1},j_{2}}\}_{R\in\D}}_{\ell^{\frac{\Q}{\Q-1},1}}
\lesssim \Norm{\{a_{R}\}_{R\in\D}}_{\ell^{\frac{\Q}{\Q-1},1}}.
\end{align}

For $\Pi_{2}(f)$, write
\begin{align*}
\EE_{\wh{R}}^{k_{2}}\otimes\HH_{R}^{k_{1}}\Big(K\rp(\unit_{\wh{R}}\otimes\unit_{R})\Big)(\wh{x},x)
&=\sum_{\substack{Q\in\D\\Q\subset R}}\sum_{j\in\Gamma_{\Q}^{k_{1}}}\int_{Q}\EE_{\wh{R}}^{k_{2}}(K(\cdot,\xi)\rp)(\wh{x}) h_{Q,k_{1}}^{j}(\xi)d\xi h_{Q,k_{1}}^{j}(x) \unit_{\wh{R}}(\wh{x})\\
&=\EE_{\wh{R}}^{k_{2}}(K(\cdot,x)\rp)(\wh{x})\unit_{R}(x)-
\EE_{\wh{R}}^{k_{2}}\otimes\EE_{R}^{k_{1}}\left(K\rp(\unit_{\wh{R}}\otimes\unit_{R})\right)(\wh{x},x),
\end{align*}
where we have used \eqref{Haar-expansion} in the second equality. Therefore,
\begin{align*}
\Pi_{2}(f)(x)=\sum_{R\in\D}a_{R}\frac{\eta_{R}(x)}{|R|^{2}}
\int_{\GG}\EE_{\wh{R}}^{k_{2}}(K(\cdot,x)\rp)(\wh{x})\unit_{R}(x) f(\wh{x})d\wh{x}-\Pi_{1}(f)(x).
\end{align*}
The estimate of $\Pi_{1}$ is done, so it remains to deal with the summation on the right hand side.
By \eqref{EQm}, we have
\begin{align*}
\EE_{\wh{R}}^{k_{2}}(K(\cdot,x)\rp)(\wh{x})\unit_{R}(x)
=\sum_{j_{2}=1}^{N_{d,m}}\int_{\wh{R}}K(\xi,x)\rp\omega_{\wh{R},k_{2}}^{j_{2}}(\xi) d\xi\omega_{\wh{R},k_{2}}^{j_{2}}(\wh{x})\unit_{R}(x).
\end{align*}
Denote
\begin{align*}
\zeta_{2,R,j_{2}}(x)=|R|^{-\frac{3}{2}}\int_{\wh{R}}K(\xi,x)\rp\omega_{\wh{R},k_{2}}^{j_{2}}(\xi) d\xi\unit_{R}(x).
\end{align*}
Then $\Pi_{2}(f)$ is reorganised as
\begin{align*}
\Pi_{2}(f)(x)=
\sum_{j_{2}=1}^{N_{d,m}}\sum_{R\in\D}a_{R}\frac{\zeta_{2,R,j_{2}}(x)\eta_{R}(x)}{|R|^{\frac{1}{2}}}
\int_{\GG}\omega_{\wh{R},k_{2}}^{j_{2}}(\wh{x}) f(\wh{x})d\wh{x}-\Pi_{1}(f)(x).
\end{align*}
By \eqref{basis-EQm} and $|K\rp(\unit_{\wh{R}}\otimes\unit_{R})(x,y)|\lesssim|R|$, we obtain $|\zeta_{2,R,j_{2}}(x)|\lesssim1$.
Moreover, $\set{\frac{\zeta_{2,R,j_{2}}\eta_{R}}{|R|^{\frac{1}{2}}}}_{R\in\D}$ and $\set{\omega_{\wh{R},k_{2}}^{j_{2}}}_{R\in\D}$ are NWO sequences.
It follows from \Cref{compact-upper-bd} and inequality \eqref{Gamma-1} that
\begin{align}\label{Gamma-2}
\Norm{\Pi_{2}}_{\L_{\frac{\Q}{\Q-1},1}}
\lesssim \sum_{j_{2}=1}^{N_{d,m}}\Norm{\{a_{R}\}_{R\in\D}}_{\ell^{\frac{\Q}{\Q-1},1}}
+\Norm{\Pi_{1}}_{\L_{\frac{\Q}{\Q-1},1}}
\lesssim\Norm{\{a_{R}\}_{R\in\D}}_{\ell^{\frac{\Q}{\Q-1},1}}.
\end{align}   
Observe that
\begin{align*}
\HH_{\wh{R}}^{k_{2}}\otimes\EE_{R}^{k_{1}}\Big(K\rp(\unit_{\wh{R}}\otimes\unit_{R})\Big)(\wh{x},x)
=\sum_{\substack{\wh{Q}\in\D\\\wh{Q}\subset \wh{R}}}\sum_{\wh{j}\in\Gamma_{\Q}^{k_{2}}}\int_{\wh{Q}}\EE_{R}^{k_{1}}(K(\xi,\cdot)\rp)(x) h_{\wh{Q},k_{2}}^{\wh{j}}(\xi)d\xi~ h_{\wh{Q},k_{2}}^{\wh{j}}(\wh{x}).
\end{align*}
By a similar estimate of $\Pi_{2}$ for $\Pi_{3}$, we have
\begin{align}\label{Gamma-3}
\Norm{\Pi_{3}}_{\L_{\frac{\Q}{\Q-1},1}}
\lesssim \Norm{\{a_{R}\}_{R\in\D}}_{\ell^{\frac{\Q}{\Q-1},1}}.
\end{align}

For $\Pi_{4}(f)$, denote
\begin{align*}
\Psi_{R,Q,j}(x)=\left(\frac{|Q|}{|R|}\right)^{\frac{1}{2}}\eta_{R}(x) h_{Q,k_{1}}^{j}(x)
\quad\mbox{and}\quad
\wh{\Psi}_{R,\wh{Q},\wh{j}}(y)=\left(\frac{|Q|}{|R|}\right)^{\frac{1}{2}}h_{\wh{Q},k_{2}}^{\tilde j}(y)\unit_{\wh{R}}(y).
\end{align*}
Clearly, $\set{\Psi_{R,Q,j}}_{R\in\D}$ and $\set{\wh{\Psi}_{R,\wh{Q},\wh{j}}}_{R\in\D}$ are NWO sequences.
Note that
\begin{align*}
\HH_{\wh{R}}^{k_{2}}\otimes\HH_{R}^{k_{1}}\Big(K\rp(\unit_{\wh{R}}\otimes\unit_{R})\Big)(\wh{x},x)=
\sum_{\substack{Q\in\D\\Q\subset R}}\sum_{j\in\Gamma_{\Q}^{k_{1}}}\sum_{\substack{\tilde Q\in\D\\\wh{Q}\subset\wh{R}}}\sum_{\wh{j}\in\Gamma_{\Q}^{k_{2}}} C_{Q,j,\tilde Q,\tilde j}h_{\wh{Q},k_{2}}^{\tilde j}(\tilde x) h_{Q,k_{1}}^{j}(x),
\end{align*}
where
$$C_{Q,j,\tilde Q,\tilde j}= \int_{\GG}\int_{\GG}\unit_{\wh{R}}(\wh{x})K(\wh{x},x)\rp\unit_{R}(x) h_{\wh{Q},k_{2}}^{\tilde j}(\tilde x) h_{Q,k_{1}}^{j}(x) dxd\tilde x.$$ 
For $R\in\D$, let $R$ be in the $l_{R}$-th dyadic cube system. Then we have
\begin{align}\label{Gmma4-expansion}
\nonumber\Pi_{4}(f)(x)
&=\sum_{j\in\Gamma_{\Q}^{k_{1}},\wh{j}\in\Gamma_{\Q}^{k_{2}}}\sum_{R\in\D}
\sum_{\substack{Q\in\D\\Q\subset R}}\sum_{\substack{\tilde Q\in\D\\\wh{Q}\subset\wh{R}}} 
a_{R} \frac{C_{Q,j,\wh{Q},\wh{j}}}{|R|^{2}} \eta_{R}(x) h_{Q,k_{1}}^{j}(x)\inner{f,h_{\wh{Q},k_{2}}^{\tilde j}}\\
\nonumber&=\sum_{j\in\Gamma_{\Q}^{k_{1}},\wh{j}\in\Gamma_{\Q}^{k_{2}}}\sum_{R\in\D}
\sum_{\substack{Q\in\D\\Q\subset R}}\sum_{\substack{\tilde Q\in\D\\\wh{Q}\subset\wh{R}}} 
a_{R} \frac{C_{Q,j,\wh{Q},\wh{j}}}{|R||Q|^{\frac{1}{2}}|\wh{Q}|^{\frac{1}{2}}} \Psi_{R,Q,j}(x)\inner{f,\wh{\Psi}_{R,\wh{Q},\wh{j}}}\\
&=\sum_{j\in\Gamma_{\Q}^{k_{1}},\wh{j}\in\Gamma_{\Q}^{k_{2}}}\sum_{R\in\D}a_{R}
\sum_{l_{1},l_{2}=l_{R}}^{+\infty}\frac{1}{2^{(l_{1}-l_{R})\Q}}\sum_{\substack{Q\in\D_{l_{1}}\\Q\subset R}}
\frac{1}{2^{(l_{1}-l_{R})\Q}}\sum_{\substack{\tilde Q\in\D_{l_{2}}\\\wh{Q}\subset\wh{R}}} 
\frac{2^{(l_{1}-l_{R})\Q}2^{(l_{2}-l_{R})\Q}C_{Q,j,\wh{Q},\wh{j}}}{|R||Q|^{\frac{1}{2}}|\wh{Q}|^{\frac{1}{2}}} \Psi_{R,Q,j}(x)\inner{f,\wh{\Psi}_{R,\wh{Q},\wh{j}}}.
\end{align}
In order to apply \Cref{compact-upper-bd-sequence}, let us investigate the decay of the main terms in \eqref{Gmma4-expansion}. 
Applying the left Taylor expansion formula \cite[Chapter 20.3.3]{BLU2007} (or \cite[(1.44) Corollary]{FS1982}) to $K(\wh{x},x)\rp$ with respect to $\wh{x}$ at point $c_{\wh{Q}}$ and then to $x$ at point $c_{Q}$ yields
\begin{align*}
K(\wh{x},x)\rp
&=P_{Q}^{(k_{2})}(\wh{x},x)+P_{Q}^{(k_{2},k_{1})}(\wh{x},x)+R_{Q}^{(k_{2},k_{1})}(\wh{x},x),
\end{align*}
where $P_{Q}^{(k_{2})}(\wh{x},x)$ is the product of a polynomial of $\wh{x}$ of homogeneous degree not greater than $k_{2}-1$ and a continuous function of $x$, and $P_{Q}^{(k_{2},k_{1})}(\wh{x},x)$ is the product of a polynomial of $\wh{x}$ (of homogeneous degree equal $k_{2}$) and a polynomial of $x$ (of homogeneous degree not greater than $k_{1}-1$), and the remainder $R_{Q}^{(k_{2},k_{1})}(\wh{x},x)$ satisfies the estimate
\begin{align*}
|R_{Q}^{(k_{2},k_{1})}(\wh{x},x)|\lesssim
\frac{1}{\rho(c_{Q},c_{\wh{Q}})^{-\Q+k_{1}+k_{2}}}\quad\mbox{for}\quad \wh{x}\in\wh{Q},x\in Q.
\end{align*}
Therefore, by \eqref{moment conditions} and \eqref{Haar size conditions} in Theorem \ref{Alpert-bases}, we have
\begin{align*}
C_{Q,j,\tilde Q,\tilde j}
&=\int_{\wh{Q}}\int_{Q}\left(K(\wh{x},x)\rp-P_{Q}^{(k_{2})}(\wh{x},x)-P_{Q}^{(k_{2},k_{1})}(\wh{x},x)\right)
h_{\wh{Q},k_{2}}^{\wh{j}}(\wh{x})h_{Q,k_{1}}^{j}(x)dx d\wh{x}\\
&=\int_{\wh{Q}}\int_{Q}R_{Q}^{(k_{2},k_{1})}(\wh{x},x)
h_{\wh{Q},k_{2}}^{\wh{j}}(\wh{x})h_{Q,k_{1}}^{j}(x)dx d\wh{x}\\
&\lesssim\int_{\wh{Q}}\int_{Q}\frac{\rho(\wh{x},c_{\wh{Q}})^{k_{2}}\rho(x,c_{Q})^{k_{1}}}{\rho(c_{Q},c_{\wh{Q}})^{-\Q+k_{1}+k_{2}}}
|Q|^{-\frac{1}{2}}|\wh{Q}|^{-\frac{1}{2}}dx d\wh{x}\\
&\lesssim\frac{2^{-l_{1}k_{1}}2^{-l_{2}k_{2}}|Q|^{\frac{1}{2}}|\wh{Q}|^{\frac{1}{2}}}{2^{-l_{R}(-\Q+k_{1}+k_{2})}}.
\end{align*}
Moreover,
\begin{align*}
\frac{2^{(l_{1}-l_{R})\Q}2^{(l_{2}-l_{R})\Q}C_{Q,j,\wh{Q},\wh{j}}}{|R||Q|^{\frac{1}{2}}|\wh{Q}|^{\frac{1}{2}}}
\lesssim2^{(l_{R}-l_{1})(k_{1}-\Q)}2^{(l_{R}-l_{2})(k_{2}-\Q)}.
\end{align*}
This provides a convergent factor for $\displaystyle\sum_{l_{1},l_{2}=l_{R}}^{+\infty}$ with $k_{1},k_{2}>\Q$.
Thus, by \eqref{Gmma4-expansion} and \Cref{compact-upper-bd-sequence}, we have
\begin{align}\label{Gamma-4}
\Norm{\Pi_{4}}_{\L_{\frac{\Q}{\Q-1},1}}
\lesssim \Norm{\{a_{R}\}_{R\in\D}}_{\ell^{\frac{\Q}{\Q-1},1}}.
\end{align}

Lastly, combining dual inequality \eqref{duality-trace}, and estimates \eqref{Gamma-1}, \eqref{Gamma-2}, \eqref{Gamma-3} and \eqref{Gamma-4}, we obtain
\begin{align*}
\Norm{\set{\Tr\Big([M_{b},T]L_{R}\Big)}_{R\in\D}}_{\ell^{\Q,\infty}}
&\leq\Norm{[M_{b},T]}_{\L_{\Q,\infty}}\sup_{\Norm{\{a_{R}\}_{R\in\D}}_{\ell^{\frac{\Q}{\Q-1},1}}\leq1}
\Norm{\sum_{R\in\D}L_{R}a_{R}}_{\L_{\frac{\Q}{\Q-1},1}}\\
&\leq\Norm{[M_{b},T]}_{\L_{\Q,\infty}}\sup_{\Norm{\{a_{R}\}_{R\in\D}}_{\ell^{\frac{\Q}{\Q-1},1}}\leq1}
\sum_{i=1}^{4}\Norm{\Pi_{i}}_{\L_{\frac{\Q}{\Q-1},1}}\\
&\leq C\Norm{[M_{b},T]}_{\L_{\Q,\infty}}.
\end{align*}
By \Cref{sobolev-osc} and \eqref{sobolev-estimate}, we conclude that
\begin{align*}
\Norm{b}_{\swq}
&\leq c_{\GG}\rp\Norm{\set{\fint_{R}|b(x)-b_{R}|dx}_{R\in\D}}_{\ell^{\Q,\infty}}\\
&\leq c_{\GG}\rp\Norm{\set{\Tr\Big([M_{b},T]L_{R}\Big)}_{R\in\D}}_{\ell^{\Q,\infty}}\\
&\leq 2C c_{\GG}\rp \Norm{[M_{b},T]}_{\L_{\Q,\infty}}.
\end{align*}
This is the desired result and the proof of necessity is completed.
\end{proof}

\bigskip
\section{Applications to quaternionic Siegel upper half space and quaternionic  Heisenberg groups }\label{applications}

The theory of slice regular functions of one quaternionic variable has been studied intensively (cf. e.g. \cite{CDLWW,CDLWW2,CGSS} and references therein)
and applied successfully to the study of quaternionic  closed operators, quaternionic function spaces and operators on them,
e.g. quaternionic  slice   Hardy space, quaternionic de Branges space, quaternionic Hankel operator and so on. Meanwhile,
quaternionic  analysis of several variables has been developed substantially in the  last three decades.   The quaternionic  counterpart of the $\overline{\partial}$-complex  and the $k$-Cauchy--Fueter complex  is  known explicitly now.  Two fundamental Calder\'on--Zygmund operator in this setting is the Cauchy--Szeg\"o projection and the Riesz transforms on quaternionic  Heisenberg groups.

We now recall some necessary notation from \cite{CDLWW}.
Let $\mathbb H^{n }$ be the $n$-dimensional quaternion
space, which is the collection of $n$-tuples $(q_1, \ldots  , q_n), q_l  \in \mathbb{H}$. We write
$q_l = x_{4l-3} + x_{4l-2}\mathbf i + x_{4l-1}\mathbf j + x_{4l}\mathbf  k$, $l = 1, \ldots  , n$.
The quaternionic  Siegel upper half space  is
$\mathcal U :=\left\{q=(q_1,\ldots, q_n)=(q_1,q')\in\mathbb H^n\mid \operatorname{Re} q_1>|q'|^2\right\}$
, whose boundary
\begin{equation*}\label{eq:boundary}
 \partial \mathcal U :=\{(q_1,q')\in\mathbb H^n\mid\rho:= {\rm Re}\,q_1-|q'|^2=0\}
\end{equation*}
 is a quadratic hypersurface.
  A $C^1$-smooth function $f = f_1 +\mathbf{i }f_2 +\mathbf{j }f_3 +\mathbf{k} f_4 : \mathcal U\rightarrow \mathbb{ H} $ is
called {\it (left) regular} on $\mathcal U  $ if it satisfies the Cauchy--Fueter equations
$\overline{ \partial}_{q_l} f (q) = 0,~  l = 1, \ldots , n$, for any $q\in \mathcal U  $,
where
\begin{equation*}\label{eq:CF}
   \overline{\partial}_{q_l}:=\partial_{  x_{4l-3}} + \partial_{ x_{4l-2}} \mathbf i + \partial_{ x_{4l-1}} \mathbf j + \partial_{ x_{4l}} \mathbf  k.
\end{equation*}

 The Hardy space $ H^p(\mathcal U )$ consists of all regular functions $F$ on $\mathcal U $, for which
$$\|F\|_{  H^p(\mathcal U )}:=\left(\sup_{\varepsilon>0}\int_{\partial \mathcal U }|F_\varepsilon(q)|^pd\beta(q)\right)^{1\over p}<\infty,$$
 where $F_\varepsilon$ is for its ``vertical translate'', i.e. $F_\varepsilon (q) = F(q + \varepsilon \bf{e_1})$,
where $ {\bf e_1} = (1, 0, 0, \ldots  , 0)$.
 The Cauchy--Szeg\"{o} projection is the  operator from   $L^2(\partial\mathcal U )$ to  $ H^2(\mathcal U )$
  satisfying the following reproducing
formula:
\begin{equation*}\label{eq:Szego0}
    F(q) =
\int_{\partial\mathcal{U} }
S(q, p)F^b(p) d\beta (p),  \qquad q\in \mathcal{U} ,
 \end{equation*}whenever $F \in   H^2(\mathcal{U} )$ with the boundary value $F^b$  on $\partial \mathcal{U} $, where $S(q,p)$ is the Cauchy--Szeg\"{o} kernel:
\begin{equation*}\label{cauchy-szego}
S(q,p)=s\Big(q_1+\overline p_1-2\sum_{k=2}^n\overline p_k\Q_k\Big)
\end{equation*}
for $p= (p_1,\ldots, p_n)\in\mathcal U$, $q=(q_1,\ldots, q_n)\in\mathcal U$, and
\begin{equation*}\label{s}
s(\sigma)=c_{n-1}{\partial^{2(n-1)}\over \partial x_1^{2(n-1)}}{{\overline \sigma}\over |\sigma|^4},\quad
\sigma=x_1+x_2{\bf i}+x_3{\bf j}+x_4{\bf k}\in\mathbb H
\end{equation*}
with the real constant $c_{n-1}$ depending only on $n$  (\cite[{Theorem A}]{CMW}). The explicit formula and the size and regularity estimate of $S(q,p)$ has been exploited only very recently in \cite{CDLWW}.

The boundary $\partial \mathcal U $ can be identified with the quaternionic Heisenberg group
 $\mathscr H^{n-1}$, which is the  space $\mathbb R^{4n-1}$ equipped with the multiplication given by
\begin{equation*}\label{law}
({\bm{t}},{\bm{y}})\cdot({\bm{t}'}, {\bm{y}}')=\bigg(\bm t+\bm t'+B(\bm y, \bm y'), {\bm{y}}+{\bm{y}}' \bigg),
\end{equation*}
\color{black}
where
${\bm{t}}=(t_1, t_2, t_3)$, ${\bm{t}}'=(t'_1, t'_2, t'_3)\in\mathbb R^3$,  ${\bm{y}}=(y_1,y_2,\cdots, y_{4n-4})$, ${\bm{y}}'=(y'_1,y'_2,\cdots, y'_{4n-4})\in\mathbb R^{4n-4}$, $B(\bm y, \bm y')=(B_1(\bm y, \bm y'),B_2(\bm y, \bm y'),B_3(\bm y, \bm y'))$, and
$$B_\alpha(\bm y, \bm y')=2\sum_{l=0}^{n-2}\sum_{j,k=1}^4b_{kj}^\alpha y_{4l+k}y'_{4l+j},\quad \alpha=1,2,3$$
with
\begin{equation*}\label{eq:b}
b^1:=\left( \begin{array}{cccc}
0&1 & 0 & 0\\
-1 & 0 & 0&0\\
0 & 0 & 0&-1\\
0 & 0 &1&0
\end{array}
\right ),
\quad
b^2:=\left( \begin{array}{cccc}
0&0 &1 & 0\\
0 & 0 & 0&1\\
-1 & 0 & 0&0\\
0 &-1 & 0 &0
\end{array}
\right ),
\quad
b^3:=\left( \begin{array}{cccc}
0&0 & 0 &1\\
0 & 0 &-1&0\\
0 &1 & 0&0\\
-1 & 0 & 0 &0
\end{array}
\right ).
\end{equation*}
%

 The following $4n-1$ vector fields are left invariant on $\mathscr H^{n-1}$:
 \begin{align*}
 Y_{4l+j} &={\partial\over \partial y_{4l+j}} + 2 \sum_{\alpha=1}^3\sum_{k=1}^4 b_{kj}^\alpha y_{4l+k} {\partial\over\partial t_\alpha},\quad l=0,\ldots, n-2, ~j=1,\ldots, 4,\\
 T_\alpha&={\partial\over\partial t_\alpha},\quad \alpha=1,2,3.\nonumber
 \end{align*}
They form a basis for the Lie algebra of left-invariant vector field on $\mathscr H^{n-1}$.
The only nontrivial commutator relations are
\begin{equation*}\label{eq:Y-bracket}
   [Y_{4l+k}, Y_{4l'+j}]=4\delta_{ll'}\sum_{\alpha=1}^3b_{kj}^\alpha{\partial\over \partial t_\alpha}, \quad l, ~l'=0,\ldots ,n-2; ~j,~k=1,\ldots,  4.
\end{equation*}
 For convenience, we set $Y_{4n-4+\alpha}:=T_\alpha$,  $y_{4n-4+\alpha }:=t_\alpha$, $~\alpha=1,2,3$.  The standard sub-Laplacian on
 $\mathscr H^{n-1}$ is defined by $\triangle_H=\sum_{j=1}^{4n-4}Y_j^2$, and the Riesz transform is given by
 $$ R_j= Y_j\triangle_H^{-1/2},\qquad j=1,\ldots,4n-4.$$
 The properties of Riesz transform were studied in \cite{Zhu}.
 
 Based on our main result, we obtain the Schatten class estimates for both $[b,S]$ and $[b,R_j]$, which recovers the related result in \cite{CDLWW,CDLWW2}. 
 Moreover, our result gives the endpoint weak Schatten estimate at the critical index, which provides the missing theory in \cite{CDLWW2}.  

\section{Appendix: oscillatory characterisation of Sobolev space}\label{Appendix}

In this Appendix section, we give the oscillatory characterisation of Sobolev space on stratified Lie groups, which has been used as a key ingredient in \Cref{endponit-estimates-sect} for the proof of \Cref{endponit-estimates}. On the other hand, the oscillatory characterisation deduced here has its own interest, extending the main results by Frank \cite{F2022} for Euclidean spaces.

For $f\in L_{loc}(\GG)$, let
\begin{align*}
m_{f}(x,r)=\fint_{B(x,r)}|f(y)-f_{B(x,r)}|dy
\end{align*}
and let $\omega_{p}$ be the measure on $\GG\times\RR_{+}$ with
\begin{align*}
d\omega_{p}(x,r)=\frac{dxdr}{r^{p+1}}.
\end{align*}

\begin{lemma}\label{coordination-transform}
Let $v\in\RR^{n_{1}}$ be a non-zero vector. Then there is a positive constant $c_{\GG}$ such that
\begin{align*}
\int_{B(o,r)}|v\cdot y^{(1)}|dy\geq c_{\GG}|v|r^{\Q+1}.
\end{align*}
Here $v\cdot y^{(1)}$ denotes $\displaystyle\sum_{j=1}^{n_{1}}v_{j}y_{j}^{(1)}$.
\end{lemma}
\begin{proof}
By \eqref{norm-inequality},
\begin{align*}
\set{y\in\GG:\rho(y)<r} \supseteq  \set{y\in\GG:\rho_{\infty}(y)<C_{\rho}\rp r}.
\end{align*}
Let 
\begin{align*}
U_{pos}=
\set{\rho_{\infty}(y)<C_{\rho}\rp r:0\leq y_{1}^{(1)}<C_{\rho}\rp r;\ \ldots;\ 0\leq y_{n_{1}}^{(1)}<C_{\rho}\rp r}.
\end{align*}
Note that the area in first layer of $B(o,r)$ is symmetric.
If $v_{j}$ is negative for some $j$, we replace the $j$-th constraint condition in $U_{pos}$ with $-C_{\rho}\rp r< y_{j}^{(1)}\leq0$. 
Without loss of generality, we assume that $v_{j}$ is positive for all $j=1,2,\ldots,n_{1}$.
Therefore,
\begin{align*}
\int_{B(o,r)}|v\cdot y^{(1)}|dy
&\geq\int_{\set{y:\rho_{\infty}(y)<C_{\rho}\rp r}}|v\cdot y^{(1)}|dy\\
&\geq\int_{U_{pos}}|v\cdot y^{(1)}|dy\\
&\geq\sum_{j=1}^{n_{1}}v_{j}\int_{U_{pos}} y_{j}^{(1)}dy.
\end{align*}
Since $dy$ is the Lebesgue measure and the area of first layer is symmetric, it follows that the integral of $y_{j}^{(1)}$ on $U_{pos}$ are equal for all $j=1,2,\ldots,n_{1}$.
Then
\begin{align*}
\int_{B(o,r)}|v\cdot y^{(1)}|dy
&\geq\sum_{j=1}^{n_{1}}v_{j}\int_{U_{pos}} y_{1}^{(1)}dy_{1}^{(1)}\cdots dy_{n_{1}}^{(1)}dy_{rest}\\
&\geq 2^{d-n_{1}-1}C_{\rho}^{-\Q-1}|v|r^{\Q+1}
\end{align*}
where we decompose $dy=dy_{1}^{(1)}\cdots dy_{n_{1}}^{(1)}dy_{rest}$.
Letting $c_{\GG}=2^{d-n_{1}}C_{\rho}^{-\Q-1}$ gives the desired result.
\end{proof}

\begin{lemma}\label{rearrangement-liminf}
Let $f\in C^{1}(\GG)$ and $0<p<\infty$.
If $\nabla f$ is Lipschitz and compactly supported, then there is a positive constant $c_{\GG,p}$ such that
\begin{align*}
\liminf_{\delta\rightarrow0^{+}}\delta^{p}\omega_{p}\Big(\{m_{f}>\delta\}\cap(\Omega\times\RR_{+})\Big)
\geq c_{\GG,p}\int_{\Omega}|\nabla f(x)|^{p}dx.
\end{align*}
\end{lemma}
\begin{proof}
Combing \Cref{coordination-transform} and the proof in \cite[Lemma 6]{F2022}, we obtain the desired result.
Here, we only adapt the one-side proof in \cite[Lemma 6]{F2022} to estimate the lower bound.
\end{proof}

For $f\in L_{loc}(\GG)$, we call $m_{f}\in L_{weak}^{p}(\GG\times\RR_{+},d\omega_{p})$ if $\displaystyle\sup_{\delta>0}\delta\omega_{p}\Big(\{m_{f} >\delta\}\Big)^{\frac{1}{p}}<\infty$.
This is a quasi-Banach space for $1< p<\infty$. Moreover, we have the following characterisation.
\begin{theorem}\label{osc-sobolev}
Let $1<p<\infty$ and $f\in L_{loc}^{1}(\GG)$. Then $f\in\hw{p}$ if and only if $m_{f}\in L_{weak}^{p}(\GG\times\RR_{+},d\omega_{p})$. Moreover,
\begin{align*}
\Norm{\nabla f}_{\lp{p}}^{p}\simeq\sup_{\delta>0}\delta^{p}\omega_{p}\Big(\{m_{f}>\delta\}\Big).
\end{align*}
\end{theorem}
\begin{proof}
{\bf Proof of necessity.}
It suffices to show 
$\displaystyle\sup_{\delta>0}\delta^{p} \omega_{p}\Big(\{m_{f}>\delta\}\Big)\lesssim\Norm{\nabla f}_{\lp{p}}^{p}$.
Using stratified mean value theorem \cite[(1.41)]{FS1982},
\begin{align*}
m_{f}(x,r)=\fint_{B(x,r)}\Big|\fint_{B(x,r)}(f(x)-f(y))dy\Big|dx
\lesssim r\fint_{B(x,r)}|\nabla f(y)|dy.
\end{align*}
By the boundedness of maximal function on $\lp{p}$ (see  e.g. \cite{FS1982}) and the proof in First part of \cite[Theorem 1]{F2022}, we obtain the desired result.

{\bf Proof of sufficiency.}
It suffices to show 
$\displaystyle\Norm{\nabla f}_{\lp{p}}^{p}\lesssim\sup_{\delta>0}\delta^{p} \omega_{p}\Big(\{m_{f}>\delta\}\Big)$. 
Let $\Omega\subset\GG$ be a bounded open set and take non-negative $\varphi\in C_{c}^{2}(\GG)$ with $\int_{\GG}\varphi dx=1$.
Following \cite[Lemma 7]{F2022}, there is a positive constant $C_{\GG,p}$ such that
\begin{align*}
\sup_{\delta>0}\delta^{p}\omega_{p}\Big(\{m_{\varphi*f} >\delta\}\Big)
\leq C_{\GG,p}\sup_{\delta>0}\delta^{p}\omega_{p}\Big(\{m_{f} >\delta\}\Big).
\end{align*}
Therefore, by \Cref{rearrangement-liminf},
\begin{align*}
c_{\GG,p}\int_{\Omega}|\nabla(\varphi_{t}*f)(x)|^{p}dx
&\leq
\liminf_{\delta\rightarrow0^{+}}\delta^{p}\omega_{p}\Big(\{m_{\varphi_{t}*f}>\delta\}\cap(\Omega\times\RR_{+})\Big)\\
&\leq\sup_{\delta>0}\delta^{p}\omega_{p}\Big(\{m_{\varphi_{t}*f}>\delta\}\Big)\\
&\leq C_{\GG,p}\sup_{\delta>0}\delta^{p}\omega_{p}\Big(\{m_{f}>\delta\}\Big).
\end{align*}
Following the proof in Second part of \cite[Theorem 1]{F2022}, we obtain the desired result.
\end{proof}

By \Cref{osc-sobolev} and \cite[Remark 4]{F2022}, we immediately deduce the following corollary, which gives the equivalent characterisation of Sobolev space in terms of mean oscillation.
\begin{corollary}\label{sobolev-osc}
Let $b\in L_{loc}(\GG)$. Then $\displaystyle\set{\fint_{R}|b(x)-b_{R}|dx}_{R\in\D}\in \ell^{\Q,\infty}$ if and only if $b\in\swq$.
Moreover, then there are constants $c_{\GG}$ and $C_{\GG}$ such that
\begin{align*}
c_{\GG}\Norm{b}_{\swq}
\leq \Norm{\set{\fint_{R}|b(x)-b_{R}|dx}_{R\in\D}}_{\ell^{\Q,\infty}}
\end{align*}
and for $1\leq q<\infty$
\begin{align*}
\Norm{\set{\Big(\fint_{R}|b(x)-b_{R}|^{q}dx\Big)^{\frac{1}{q}}}_{R\in\D}}_{\ell^{\Q,\infty}}
\leq C_{\GG}\Norm{b}_{\swq}.
\end{align*}
\end{corollary}

\bigskip
\bigskip
{\bf Acknowledgements:} Ji Li is supported by ARC DP 220100285. X. Xiong and F. L. Yang are supported by National Natural Science Foundation of China, General Project 12371138.

\end{document}